\documentclass{amsart}
\usepackage{amsmath,amssymb,graphicx,latexsym}
\usepackage{mathdots}		
\usepackage{stmaryrd}		
\usepackage{tikz}
\usepackage{tikz-cd}
\usepackage{hyperref}

\usepackage{caption}
\usepackage{subcaption}

\newtheorem{theorem}{Theorem}
\newtheorem{conjecture}[theorem]{Conjecture}
\newtheorem{corollary}[theorem]{Corollary}
\newtheorem{lemma}[theorem]{Lemma}
\newtheorem{proposition}[theorem]{Proposition}
\newtheorem*{main theorem}{Theorem~\ref{kernel size upper bound}}
\newtheorem*{main theorem - asymptotic}{Theorem~\ref{kernel size asymptotic bound}}

\theoremstyle{definition}
\newtheorem{example}[theorem]{Example}
\newtheorem*{definition*}{Definition}
\newtheorem*{notation*}{Notation}
\newtheorem{remark}[theorem]{Remark}

\DeclareMathOperator{\lcm}{lcm}
\DeclareMathOperator{\orb}{orb}
\DeclareMathOperator{\partitions}{parts}
\DeclareMathOperator{\rad}{rad}

\newcommand{\F}{\mathbb{F}}

\newcommand{\Z}{\mathbb{Z}}

\newcommand{\field}{\F_q}

\newcommand{\Landau}{g}
\newcommand{\Landaulcm}{\mathcal L}

\newcommand{\dd}{\textnormal{d}}

\newcommand{\subl}{_\ell}
\newcommand{\subr}{_\textnormal{r}}
\newcommand{\subt}{_\textnormal{t}}
\newcommand{\subb}{_\textnormal{b}}

\newcommand{\nequiv}{\mathrel{\not\equiv}}
\newcommand{\colonequal}{\mathrel{\mathop:}=}

\newcommand{\size}[1]{\lvert{#1}\rvert}
\newcommand{\ceil}[1]{\left\lceil #1 \right\rceil}
\newcommand{\doublebracket}[1]{\llbracket #1 \rrbracket}
\newcommand{\floor}[1]{\left\lfloor #1 \right\rfloor}
\newcommand{\paren}[1]{\left( #1 \right)}

\newcommand{\seq}[1]{\cite[\href{http://oeis.org/#1}{#1}]{OEIS}}

\begin{document}

\title{An elementary proof of Bridy's theorem}

\author{Eric Rowland}
\address{
	Department of Mathematics \\
	Hofstra University \\
	Hempstead, NY \\
	USA
}
\author{Manon Stipulanti}
\address{
	Department of Mathematics \\
	University of Li\`ege\\
	All\'ee de la D\'ecouverte 12\\
	4000 Li\`ege, Belgium
}
\author{Reem Yassawi}
\address{
	School of Mathematical Sciences \\
	Queen Mary University of London \\
	Mile End Road \\
	London E1 4NS \\
	UK
}

\date{February 2, 2025}


\thanks{The second author is an FNRS Research Associate supported by the Research grant
1.C.104.24F.
The third author was supported by the EPSRC, grant number EP/V007459/2.}

\begin{abstract}
Christol's theorem states that a power series with coefficients in a finite field is algebraic if and only if its coefficient sequence is automatic.
A natural question is how the size of a polynomial describing such a sequence relates to the size of an automaton describing the same sequence.
Bridy used tools from algebraic geometry to bound the size of the minimal automaton for a sequence, given its minimal polynomial.
We produce a new proof of Bridy's bound by embedding algebraic sequences as diagonals of rational functions.
\end{abstract}

\maketitle

\section{Introduction}\label{section: introduction}

A well-known result of Christol~\cite{Christol, Christol--Kamae--Mendes France--Rauzy} states that a sequence $a(n)_{n \geq 0}$ of elements in the finite field $\field$ is algebraic if and only if it is $q$-automatic.
That is, its generating series $F = \sum_{n \geq 0} a(n) x^n$ satisfies $P(x, F) = 0$ for some nonzero polynomial $P \in \field[x, y]$ precisely when there exists a finite automaton that outputs $a(n)$ when fed the standard base-$q$ representation of $n$.
Such sequences can therefore be represented both by polynomials and by automata.
A natural question is how the size of the automaton, measured by the number of states, depends on the size of the polynomial, measured by its \emph{height} $h \colonequal \deg_x P$ and \emph{degree} $d \colonequal \deg_y P$.
Using tools from algebraic geometry, Bridy~\cite{Bridy} showed that the number of states is in $(1 + o(1)) q^{h d}$ as $q$, $h$, or $d$ tends to infinity.

In this paper, we give a new proof of Bridy's theorem using tools from linear algebra and results about constant-recursive sequences.
Quite apart from the interest of providing an elementary proof of Bridy's result, our approach generalizes to settings that are not accessible to algebraic geometry.
An analogue of Christol's theorem has been established for sequences of $p$-adic integers~\cite{Christol Fonctions, Denef--Lipshitz}.
In a subsequent paper~\cite{Rowland--Yassawi algebraic series II}, we use our approach to bound the number of states in the minimal automaton for an algebraic sequence of $p$-adic integers reduced modulo~$p^\alpha$.

All automata in this article read representations of integers starting with the least significant digit; see Section~\ref{section: vector space}.
We will be interested in sequences with polynomial representations as follows.

\begin{definition*}
Let $P \in \field[x,y]$ such that $P(0, 0) = 0$ and $\frac{\partial P}{\partial y}(0, 0) \neq 0$.
The \emph{Furstenberg series} associated with $P$ is the unique power series $F \in \field\doublebracket{x}$ satisfying $F(0) = 0$ and $P(x, F) = 0$.
\end{definition*}

The condition $\frac{\partial P}{\partial y}(0, 0) \neq 0$ is a statement about the coefficient of $x^0 y^1$.
It guarantees that $d \geq 1$.
If $h = 0$, then $F$ is the trivial $0$ series, so we may assume $h \geq 1$.
Along with the condition $P(0, 0) = 0$, a version of the implicit function theorem guarantees the uniqueness of $F$~\cite[Theorem 2.9]{Kauers--Paule}.
Given a polynomial $P$ which does not satisfy the conditions $P(0, 0) = 0$ and $\frac{\partial P}{\partial y}(0, 0) \neq 0$, and a power series $F$ satisfying $P(x, F) = 0$, there is a technique to obtain a polynomial $\bar{P}$ and a ``shift'' $\bar{F}$ of $F$ such that
$\bar{F}$ is the Furstenberg series associated with $\bar{P}$.
For example, see~\cite[Lemma~6.2]{Adamczewski--Bell diagonalization} for details.
The results in this article are stated for Furstenberg series, but this technique can be used to extend them to general algebraic series.

Our main result is Theorem~\ref{kernel size upper bound}, whose statement needs a few definitions.
Define $\partitions(n)$ to be the set of all integer partitions of $n$.
We are interested in the lcm of an integer partition, since it will arise as $\lcm(\deg R_1, \dots, \deg R_k)$ where $R_1, \dots, R_k$ are the irreducible factors of a polynomial of fixed degree $n$.
The \emph{Landau function} $\Landau(n)$ outputs the maximum value of $\lcm(\sigma)$ over all integer partitions $\sigma \in \partitions(n)$~\seq{A000793}.
For example, $\Landau(5)$ is the maximum value among $\lcm(5)$, $\lcm(4, 1)$, $\lcm(3, 2)$, $\lcm(3, 1, 1)$, $\lcm(2, 2, 1)$, $\lcm(2, 1, 1, 1)$, and $\lcm(1, 1, 1, 1, 1)$, so we have $\Landau(5) = 6$.
The Landau function also appeared in Bridy's analysis.
We will use a variant of the Landau function that gives a better bound.
Define
\[
	\Landaulcm(l, m, n) \colonequal
	\max_{\substack{
		\vphantom{\sigma_1 \in \partitions(i)} 1 \leq i \leq l \\
		\vphantom{\sigma_2 \in \partitions(j)} 1 \leq j \leq m \\
		\vphantom{\sigma_3 \in \partitions(k)} 1 \leq k \leq n
	}}
	\max_{\substack{
		\sigma_1 \in \partitions(i) \\
		\sigma_2 \in \partitions(j) \\
		\sigma_3 \in \partitions(k)
	}}
	\lcm(\lcm(\sigma_1), \lcm(\sigma_2), \lcm(\sigma_3)).
\]

\begin{theorem}\label{kernel size upper bound}
Let $F = \sum_{n \geq 0} a(n) x^n \in \field\doublebracket{x} \setminus \{0\}$ be the Furstenberg series associated with a polynomial $P \in \field[x, y]$ of height $h$ and degree $d$.
Then the minimal $q$-automaton that generates $a(n)_{n\geq 0}$ has size at most
\[
	q^{h d} + q^{(h - 1) (d - 1)} \Landaulcm(h, d, d) + \floor{\log_q h} + \floor{\log_q \max(h, d)} + 3.
\]
\end{theorem}

In Section~\ref{section: numeric examples}, we give numeric evidence that the bound in Theorem~\ref{kernel size upper bound} is asymptotically sharp.
As a corollary of Theorem~\ref{kernel size upper bound}, we obtain Bridy's theorem~\cite{Bridy}.

\begin{theorem}\label{kernel size asymptotic bound}
Let $F = \sum_{n \geq 0} a(n) x^n \in \field\doublebracket{x}$ be the Furstenberg series associated with a polynomial $P \in \field[x, y]$ of height $h$ and degree $d$.
Then the size of the minimal $q$-automaton generating $a(n)_{n\geq 0}$ is in $(1 + o(1)) q^{h d}$ as any of $q$, $h$, or $d$ tends to infinity and the others remain constant.
\end{theorem}

Bridy~\cite{Bridy} also showed that the number of states is in $(1 + o(1)) q^{h + d + g - 1}$ as any of $q, h, d, g$ tends to infinity, where $g$ is the genus of $P$.
Since the genus satisfies $g \leq (h - 1) (d - 1)$, Bridy obtains the bound $(1 + o(1)) q^{h d}$ for the number of states.
Let $G$ be the number of interior points in the Newton polygon of $P$.
We have $g \leq G$ by Baker's theorem~\cite{Beelen}, with equality generically.
In our setting, one could use $G$ to obtain more refined bounds than in Theorem~\ref{kernel size upper bound}, analogous to Bridy's bound.
This approach is discussed briefly in~\cite[Section 6]{Adamczewski--Yassawi}.

Broadly, the proof of Theorem~\ref{kernel size upper bound} consists of two steps.
First, in Section~\ref{section: vector space}, we represent states in the automaton with bivariate polynomials, and we establish basic properties of a space $W$ of bivariate polynomials containing most of the automaton's states.
Namely, $W$ contains all states except those in the orbit $\orb_{\lambda_{0,0}}(S_0)$ of the initial state $S_0$ under the linear transformation $\lambda_{0, 0}(S) \colonequal \Lambda_{0, 0}\!\paren{S Q^{q - 1}}$, where $\Lambda_{0, 0}$ is a Cartier operator and $Q = P / y$.
The space $W$ has size $q^{h d}$, giving the main term in Theorem~\ref{kernel size upper bound}.
This first step is elementary and yields an initial upper bound of $q^{h d} + \size{\orb_{\lambda_{0,0}}(S_0)}$ for the number of states.

The second step is considerably more involved.
We show that the size of an orbit under $\lambda_{0, 0}$ is small, giving the lower-order terms in Theorem~\ref{kernel size upper bound}.
The key idea is that one can bound the orbit size under $\lambda_{0, 0}$ in terms of the orbit sizes under restrictions of $\lambda_{0, 0}$ to four subspaces.
One subspace has size $q^{(h - 1) (d - 1)}$.
On the other three, the operator $\lambda_{0, 0}$ behaves like linear transformations $\lambda_0(S) \colonequal \Lambda_0\!\paren{S R^{q - 1}}$ on univariate polynomials for certain Laurent polynomials $R$, where $\Lambda_0$ is a Cartier operator.
We show how to bound the orbit size under $\lambda_0$ in terms of the factorization of $R$, using the period length of the coefficient sequence of the series $\frac{1}{R}$.
Surprisingly, this period length is not dependent on $q$; this appears starting in Theorem~\ref{square-free orbit size upper bound - positive degree}.

Our proof of Theorem~\ref{kernel size upper bound} begins by converting the representation of the series $F$ by a polynomial $P$ to a representation as the diagonal of a rational function.
More generally, in Theorem~\ref{kernel size upper bound - diagonals} we use the same two steps to bound the automaton size for the diagonal of a rational function in two variables.
For more than two variables, new techniques would be needed to further extend the second step.
The current techniques would only give an analogue of Corollary~\ref{kernel initial upper bound}.

This analogue is already included in recent work by Adamczewski, Bostan, and Caruso~\cite{Adamczewski--Bostan--Caruso}, who bound the dimension of a vector space containing the kernel (see Section~\ref{section: vector space}) of a multidimensional algebraic sequence, generalizing a result of Bostan, Caruso, Christol, and Dumas~\cite{Bostan--Caruso--Christol--Dumas} for one-dimensional algebraic sequences.
These papers also use diagonals, and the argument fundamentally follows the lines of a multivariate version of Section~\ref{section: vector space} below.
However, like Bridy, the authors of \cite{Adamczewski--Bostan--Caruso} give a more refined bound in terms of the genus of the associated surface.
They also give several applications of their bound, establishing a polynomial bound on the algebraic degree of reductions modulo~$p$ of diagonals of multivariate algebraic power series, answering a question of Deligne~\cite{Deligne}, and improving Harase's bound~\cite{Harase II} on the degree of the Hadamard product of two algebraic power series.

In Section~\ref{section: vector space}, we lay the groundwork and obtain a preliminary, coarser bound on the size of the automaton, in Corollary~\ref{kernel preliminary upper bound}.
In Section~\ref{section: structure}, we study the linear structure of the operator $\lambda_{0,0}$ and show in Proposition~\ref{univariate emulation} that it can be emulated by univariate operators $\lambda_0$ on certain subspaces of $\field[z]$.
In Section~\ref{section: orbit size univariate}, we bound the orbit size of a polynomial under $\lambda_0$, leading to Theorem~\ref{univariate orbit size upper bound}.
Finally in Section~\ref{section: orbit size}, we tie these results together to obtain Theorem~\ref{kernel size upper bound} and an analogous result for diagonals of rational functions.
In Section~\ref{section: univariate conjectures}, we give some intriguing conjectures about orbits under $\lambda_0$ that were discovered in the process of proving Theorem~\ref{kernel size upper bound} and ultimately not used.

\section{Numeric evidence for sharpness}\label{section: numeric examples}

In this section, we systematically find Furstenberg series, represented by polynomials $P$, for which the corresponding automata are large.
The computations are performed with the Mathematica package \textsc{IntegerSequences}~\cite{IntegerSequences, IntegerSequences article}.

For fixed values of $q$, $h$, and $d$, we generate all polynomials $P \in \field[x, y]$ with height~$h$ and degree~$d$ that satisfy the conditions in the definition of a Furstenberg series.
We also require that the coefficient of $x^0 y^1$ in $P$ is $1$, since $P$ and $c P$ (where $c \neq 0$) define the same series $F$ and produce the same automaton.
Then, for each $P$, we use the construction described in Section~\ref{section: vector space} below to compute an automaton generating the coefficient sequence of its associated Furstenberg series.
In general, this construction does not produce a minimal automaton.
Minimizing is costly, so to expand the feasible search space we do not minimize automata at this step.
Instead, we determine the size of each unminimized automaton, select one of the polynomials $P$ that maximizes this size, and minimize its automaton.

Table~\ref{automaton size table} in the appendix lists the maximum unminimized automaton size for several values of $q$, $h$, and $d$, along with one polynomial that achieves this size and the value of the bound in Theorem~\ref{kernel size upper bound}.
For each polynomial in Table~\ref{automaton size table}, the automaton size drops by at most $1$ during minimization.
This justifies the decision to not minimize all automata initially.
For $d = 1$ (that is, rational series), Bridy showed that the bound $(1 + o(1)) q^{h d}$ is sharp by constructing polynomials $P$ from univariate primitive polynomials~\cite[Proposition 3.14]{Bridy}.
Table~\ref{automaton size table} suggests this bound is also sharp for $d \geq 2$.
For $d = 2$, Figure~\ref{automaton size figure} in the appendix shows the distribution of unminimized automaton sizes for some values of $q$ and $h$ by plotting the number of polynomials with each size.

Most of the article is concerned with bounding the orbit sizes of polynomials under the operator $\lambda_{0, 0}$.
This will yield the terms other than $q^{h d}$ in Theorem~\ref{kernel size upper bound}.
Table~\ref{orbit size table} lists the maximum orbit size under $\lambda_{0, 0}$ for several values of $q$, $h$, and $d$.

Whereas the polynomials in Table~\ref{automaton size table} produce automata close to the upper bound, some algebraic sequences that arise in combinatorics, when reduced modulo~$p$, are generated by rather small automata.
For example, let $C(n)$ be the $n$th Catalan number~\cite[A000108]{OEIS}.
Its generating series $F = 1 + x + 2 x^2 + 5 x^3 + \cdots$ satisfies $x F^2 - F + 1 = 0$, so $h = 1$ and $d = 2$.
Burns~\cite[Section~4]{Burns} gave an explicit construction for an automaton that generates $(C(n) \bmod p)_{n \geq 0}$.
This automaton has only $p + 3$ states, compared to the bound $p^2 + \Landaulcm(1, 2, 2) + 3 = p^2 + 5$ in Theorem~\ref{kernel size upper bound}.

\section{The vector space of possible states}\label{section: vector space}

Christol's theorem implies that an algebraic sequence of elements in $\field$ is $q$-automatic.
In this section, we establish a correspondence between states of an automaton generating such a sequence and polynomials in a finite-dimensional $\field$-vector space.
We do this by converting states in the automaton first to sequences, then to power series, and finally to polynomials.
This correspondence provides the foundation for the rest of the article, and we use it to give a preliminary upper bound on the number of states in Corollary~\ref{kernel preliminary upper bound}.

We assume the reader is familiar with deterministic finite automata with output.
See~\cite{Allouche--Shallit} for a comprehensive treatment and \cite{Rowland} for a short introduction.
An automaton with input alphabet $\{0, 1, \dots, q - 1\}$ generates the $q$-automatic sequence $a(n)_{n \geq 0}$, where $a(n)$ is the output of the automaton when fed the standard base-$q$ representation of $n$, starting with the least significant digit, i.e.\ when the digits are fed from right to left.
In general, automata are sensitive to leading $0$s; that is, the output changes when fed a nonstandard representation of $n$.
One can always produce an automaton without this drawback~\cite[Theorem~5.2.3]{Allouche--Shallit}, although the number of states may increase.

\begin{example}
The two automata
\[
	\vcenter{\hbox{\includegraphics[width = .4\textwidth]{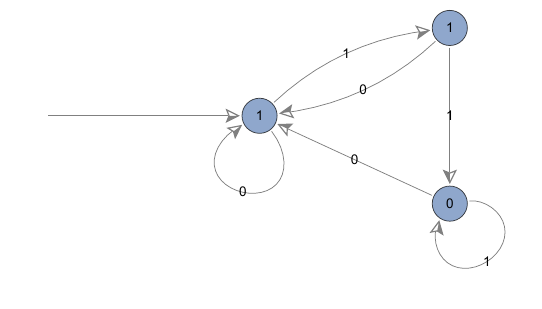}}}
	\vcenter{\hbox{\includegraphics[width = .5\textwidth]{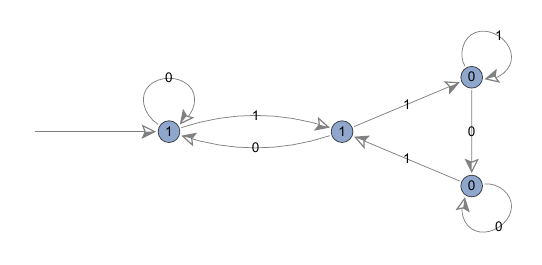}}}
\]
generate the same $2$-automatic sequence $1, 1, 1, 0, 1, 1, 0, 0, 1, 1, 1, 1, 0, 0, 0, 0, \dots$.
The behavior of the first automaton is affected by leading $0$s; for example, feeding $11$ into this automaton produces the output $0$, whereas the input $011$ produces the output $1$.
The behavior of the second automaton is not affected by leading $0$s, and in fact this is the smallest automaton with this property for this sequence.
\end{example}

Given a $q$-automatic sequence $a(n)_{n \geq 0}$, we refer to the smallest automaton that generates $a(n)_{n \geq 0}$ and that is not affected by leading $0$s as its \emph{minimal automaton}.
Theorem~\ref{kernel size upper bound} gives an upper bound on the number of states in the minimal automaton.
Theorem~\ref{kernel size upper bound} also gives an upper bound on the size of the \emph{$q$-kernel} of $a(n)_{n \geq 0}$, defined as
\[
	\ker_q(a(n)_{n \geq 0}) \colonequal \{a(q^e n + r)_{n \geq 0} : \text{$e \geq 0$ and $0 \leq r \leq q^e - 1$}\}.
\]
A sequence is $q$-automatic if and only if its $q$-kernel is finite; this is known as Eilenberg's theorem.
Moreover, the states of the minimal automaton are in bijection with the elements of the $q$-kernel.

We then represent kernel sequences $a(q^e n + r)_{n \geq 0}$ by their generating series $\sum_{n \geq 0} a(q^e n + r) x^n$.
Let $\field\doublebracket{x}$ and $\field\doublebracket{x, y}$ denote the sets of univariate and bivariate power series with coefficients in $\field$.
Analogously, $\field[x]$ and $\field[x, y]$ denote sets of polynomials.
Elements of the $q$-kernel (and therefore states in the minimal automaton) can be accessed by applying the following operators.

\begin{definition*}
Let $n \in \Z$.
For each $r \in \{0, 1, \dots, q - 1\}$, define the \emph{Cartier operator} $\Lambda_r$ on the monomial $x^n$ by
\[
	\Lambda_r(x^n) =
	\begin{cases}
		x^\frac{n - r}{q}	& \text{if $n \equiv r \mod q$} \\
		0			& \text{otherwise}.
	\end{cases}
\]
Then extend $\Lambda_r$ linearly to polynomials (as well as to Laurent polynomials and Laurent series) in $x$ with coefficients in $\field$.
In particular, for polynomials we have
\[
	\Lambda_r\!\paren{\sum_{n = 0}^N a(n) x^n}
	= \sum_{n = 0}^{\floor{N/q}} a(q n + r) x^n.
\]
Similarly, for $m, n \in \Z$ and $r, s \in \{0, 1, \dots, q - 1\}$, define the bivariate Cartier operator
\[
	\Lambda_{r, s}(x^m y^n) =
	\begin{cases}
		x^\frac{m - r}{q} y^\frac{n - s}{q}	& \text{if $m \equiv r \mod q$ and $n \equiv s \mod q$} \\
		0						& \text{otherwise},
	\end{cases}
\]
and extend $\Lambda_{r, s}$ linearly to bivariate polynomials (as well as to Laurent polynomials and Laurent series).
\end{definition*}

The map $\Lambda_r$ on $\field\doublebracket{x}$ corresponds to the map $a(n)_{n \geq 0} \mapsto a(q n + r)_{n \geq 0}$.
An advantage of representing sequences by power series is that a factor of the form $F^q$ can be pulled out of a Cartier operator, as in the following proposition.
We will use this repeatedly.
The univariate case is proved in~\cite[Lemma~12.2.2]{Allouche--Shallit} for power series; the Laurent series and bivariate cases are similar.

\begin{proposition}\label{Cartier}
If $F$ and $G$ are Laurent series in $x$ with coefficients in $\field$, then $\Lambda_r(G F^q) = \Lambda_r(G) F$.
Similarly, if $F, G \in \field\doublebracket{x, y}$, then $\Lambda_{r, s}(G F^q) = \Lambda_{r, s}(G) F$.
\end{proposition}

The final step is to use a theorem of Furstenberg~\cite{Furstenberg} to convert each algebraic power series $\sum_{n \geq 0} a(q^e n + r) x^n$ corresponding to a kernel sequence to the diagonal of a rational function.
Since different rational functions can have the same diagonal, a given kernel sequence is potentially the diagonal of several rational functions that arise, so the resulting automaton is not necessarily minimal.
However, the number of distinct rational functions that arise is an upper bound on the size of the kernel.
Furstenberg's theorem holds more generally over every field, but we state it for $\field$.
The \emph{diagonal operator} $\mathcal{D} \colon \field\doublebracket{x, y} \to \field\doublebracket{x}$ is defined by
\[
	\mathcal{D}\!\paren{\sum_{m \geq 0} \sum_{n \geq 0} a(m, n) x^m y^n}
	= \sum_{n \geq 0} a(n, n) x^n.
\]
For a bivariate series or polynomial $P$, define $P(a, b)$ to be $P|_{x \to a, y \to b}$, and similarly for univariate series;
for example, if $P = 3 x + 2 y + x y$, then $P(x y, y) = 3 x y + 2 y + x y^2$, $\frac{\partial P}{\partial y} = 2 + x$, $\frac{\partial P}{\partial y}(x y, y) = 2 + x y$, and $\frac{\partial P}{\partial y}(0, 0) = 2$.

Recall the definition of a Furstenberg series from the introduction.

\begin{theorem}[Furstenberg]\label{Furstenberg}
Let $F \in \field\doublebracket{x}$ be the Furstenberg series associated with a polynomial $P \in \field[x, y]$.
Then
\[
	F = \mathcal{D}\!\paren{\frac{
		y \frac{\partial P}{\partial y}(x y, y)
	}{
		P(x y, y) / y
	}}.
\]
\end{theorem}

The conditions $P(0, 0) = 0$ and $\frac{\partial P}{\partial y}(0, 0) \neq 0$ guarantee that every monomial in $P(x y, y)$ is divisible by $y$ and that
\begin{equation}\label{equation: Furstenberg}
	\frac{
		y \frac{\partial P}{\partial y}(x y, y)
	}{
		P(x y, y) / y
	}
\end{equation}
has a unique power series expansion.

Now applying a Cartier operator to the diagonal of a rational power series produces another diagonal of a rational power series, namely
\begin{equation}\label{emulation}
	\Lambda_r \mathcal{D}\!\paren{\frac{S}{Q}}
	= \mathcal{D} \Lambda_{r, r}\!\paren{\frac{S}{Q}}
	= \mathcal{D} \Lambda_{r, r}\!\paren{\frac{S Q^{q - 1}}{Q^q}}
	= \mathcal{D}\!\paren{\frac{\Lambda_{r, r}\!\paren{S Q^{q - 1}}}{Q}},
\end{equation}
where the last equality follows from Proposition~\ref{Cartier}.
Since the initial and final rational series in Equation~\eqref{emulation} have the same denominator $Q$, every sequence in the $q$-kernel of $a(n)_{n \geq 0}$, and hence every state in the automaton, is the diagonal of a rational function with denominator $Q$.
Therefore we can represent each state simply by its numerator, and the map $S \mapsto \Lambda_{r, r}\!\paren{S Q^{q - 1}}$ on $\field[x, y]$ emulates the Cartier operator $\Lambda_r$ on $\field\doublebracket{x}$.
Moreover, the common denominator $Q$ is the denominator of the rational expression corresponding to $a(n)_{n \geq 0}$ itself, which is $P(x y, y) / y$ by Theorem~\ref{Furstenberg}.
This is the approach taken elsewhere~\cite{Denef--Lipshitz, Adamczewski--Bell diagonalization, Rowland--Yassawi}.

However, in this article we shear the bivariate series~\eqref{equation: Furstenberg} by replacing $x$ with $x y^{-1}$, obtaining
\begin{equation}\label{equation: simple Furstenberg}
	\frac{
		y \frac{\partial P}{\partial y}
	}{
		P/y
	}
\end{equation}
instead.
The diagonal of~\eqref{equation: Furstenberg} is the $y^0$ row of~\eqref{equation: simple Furstenberg}; that is, for each $n \geq 0$, the coefficient of $x^n y^n$ in~\eqref{equation: Furstenberg} is the coefficient of $x^n y^0$ in~\eqref{equation: simple Furstenberg}.
The latter is significantly more convenient notationally for obtaining the desired bound.
Let $Q \colonequal P / y$ be the denominator.
Note that $Q$ is a polynomial in $x$, but it may be a Laurent polynomial in $y$ and not a polynomial.
This will not cause us trouble, but we mention that, to expand~\eqref{equation: simple Furstenberg} as a series and get the intended sequence of coefficients of $x^n y^0$, we should expand using the constant term of the denominator $Q$ (since it is the same as the constant term of $P(x y, y) / y$) and not a monomial involving $y^{-1}$ if present.
The series expansion of~\eqref{equation: simple Furstenberg} is a power series in $x$ but may have terms involving $y^n$ with negative $n$.

In this formulation, the diagonal operator is replaced by the \emph{center row operator} $\mathcal{C}$, defined by
\[
	\mathcal{C}\!\paren{\sum_{m \geq 0} \sum_{n \in \Z} a(m, n) x^m y^n}
	= \sum_{m \geq 0} a(m, 0) x^m.
\]
Equation~\eqref{emulation} becomes
\[
	\Lambda_r \mathcal{C}\!\paren{\frac{S}{Q}}
	= \mathcal{C} \Lambda_{r, 0}\!\paren{\frac{S}{Q}}
	= \mathcal{C} \Lambda_{r, 0}\!\paren{\frac{S Q^{q - 1}}{Q^q}}
	= \mathcal{C}\!\paren{\frac{\Lambda_{r, 0}\!\paren{S Q^{q - 1}}}{Q}}.
\]
Therefore the map $S \mapsto \Lambda_{r, 0}\!\paren{S Q^{q - 1}}$ on $\field[x, y]$ emulates the Cartier operator $\Lambda_r$ on $\field\doublebracket{x}$.
We represent each state in the automaton by a polynomial $S \in \field[x, y]$.
The initial state is $S_0 \colonequal y \frac{\partial P}{\partial y}$, since this is the numerator of the rational series corresponding to $a(n)_{n \geq 0}$.
From each state $S$, upon reading $r \in \{0, 1, \dots, q - 1\}$ we transition to $\Lambda_{r, 0}\!\paren{S Q^{q - 1}}$.
The output assigned to each state $S$ is $\frac{S(0, 0)}{Q(0,0)}$.
\begin{remark}
If $r = 0$, then the output assigned to the state $S$ is the same as the output assigned to $\Lambda_{0, 0}\!\paren{S Q^{q - 1}}$ since the constant term of $Q^{q - 1}$ is $1$ by the assumption $\frac{\partial P}{\partial y}(0, 0) \neq 0$.
Therefore, the constructed automaton is not sensitive to leading $0$s.
\end{remark}

We solidify our notation as follows.

\begin{notation*}
For the remainder of the article, we fix a prime power $q$ and a polynomial $P \in \field[x, y]$ with height~$h\geq 1$ and degree~$d\geq 1$.
We assume that $P(0, 0) = 0$ and $\frac{\partial P}{\partial y}(0, 0) \neq 0$ so that we obtain the Furstenberg series $F \in \field\doublebracket{x}$ given by Theorem~\ref{Furstenberg}.
Let $Q = P/ y$.
For each $r \in \{0, 1, \dots, q - 1\}$ and each $S \in \field[x, y]$, define
\[
	\lambda_{r, 0}(S)
	\colonequal
	\Lambda_{r, 0}\!\paren{S Q^{q - 1}}.
\]
Note that $\lambda_{r, 0}$ depends on $Q$, even though the notation does not reflect this.
Let $W$ be the $\field$-vector space defined by
\[
	W \colonequal \left\langle x^i y^j : \text{$0 \leq i \leq h - 1$ and $0 \leq j \leq d - 1$}\right\rangle.
\]
We will always use this basis of $W$.
We have $\dim W = h d$, so $\size{W} = q^{h d}$.
\end{notation*}

\begin{example}\label{example}
Let $q = 3$, and consider the polynomial
\[
	P = (x^2 + x + 2) y^4 + x y^3 + (2 x + 1) y^2 + (x^2 + 1) y + 2 x^2 + x \in \F_3[x, y]
\]
with height~$h = 2$ and degree~$d = 4$.
We will use this polynomial as a running example throughout the paper.
The coefficient sequence $a(n)_{n \geq 0}$ of the series $F \in \F_3\doublebracket{x}$ satisfying $P(x, F) = 0$ is
\[
	0, 2, 0, 2, 0, 2, 0, 0, 1, 0, 0, 1, 1, 1, 1, 1, 2, 1, 1, 2, 0, 0, 2, 2, 1, 0, 1, \dots.
\]
We have
\[
	Q = P/ y = (x^2 + x + 2) y^3 + x y^2 + (2 x + 1) y + x^2 + 1 + (2 x^2 + x) y^{-1}.
\]
The initial state is
\[
	S_0
	= y \tfrac{\partial P}{\partial y}
	= (x^2 + x + 2) y^4 + (x + 2) y^2 + (x^2 + 1) y.
\]
The space $W$ consists of all bivariate polynomials with height at most $1$ and degree at most $3$.
\end{example}

In the remainder of this section, we use elementary methods to highlight the relevance of $W$, leading us to a preliminary bound on the size of the kernel in Corollary~\ref{kernel preliminary upper bound}.

Proposition~\ref{module} shows that $W$ is closed under $\lambda_{r, 0}$.
In particular, even though $Q$ is possibly a Laurent polynomial, $\lambda_{r, 0}(S)$ is a polynomial for each $S\in W$.

\begin{proposition}\label{module}
For each $r \in \{0, 1, \dots, q - 1\}$, we have $\lambda_{r, 0}(W) \subseteq W$.
\end{proposition}

\begin{proof}
Let $c \, x^I y^J$ be a nonzero monomial in the Laurent polynomial $Q^{q - 1}$.
The height $I$ of this monomial satisfies $0 \leq I \leq (q - 1) h$, and the degree $J$ satisfies $-(q - 1) \leq J \leq (q - 1) (d - 1)$.
To show that $\lambda_{r, 0}(W) \subseteq W$, by linearity it suffices to show that if $x^i y^j$ is an element of the basis of $W$ then $\Lambda_{r, 0}\!\paren{x^i y^j \cdot x^I y^J} \in W$.
One computes
\[
	\Lambda_{r, 0}\!\paren{x^i y^j \cdot x^I y^J}
	= \begin{cases}
		x^{\frac{i + I - r}{q}} y^{\frac{j + J}{q}}	& \text{if $i+I \equiv r \bmod q$ and $j+J \equiv 0 \bmod q$}\\
		0									& \text{otherwise}.
	\end{cases}
\]
In the second case, clearly the monomial $0$ belongs to $W$.
In the first case, the height of this monomial satisfies $\frac{i + I - r}{q}\ge 0$.
We use the fact that $\frac{i + I - r}{q}$ and $\frac{j + J}{q}$ are integers.
Then
\begin{equation}\label{image height}
	\tfrac{i + I - r}{q}
	= \floor{\tfrac{i + I - r}{q}}
	\le \floor{\tfrac{(h - 1) + (q - 1) h - r}{q}}
	= \floor{\tfrac{q h - r - 1}{q}}
	\le h - 1.
\end{equation}
The degree of $\Lambda_{r, 0}\!\paren{x^i y^j \cdot x^I y^J}$ satisfies
$
	\frac{j + J}{q}
	\leq \frac{(d - 1) + (q - 1) (d - 1)}{q}
	= d - 1
$
and
\[
	\tfrac{j + J}{q}
	= \ceil{\tfrac{j + J}{q}}
	\geq \ceil{\tfrac{0 - (q - 1)}{q}}
	= \ceil{\tfrac{1}{q} - 1}
	= 0.
\]
Consequently $\Lambda_{r, 0}\!\paren{x^i y^j \cdot x^I y^J} \in W$.
\end{proof}

The initial state $y \frac{\partial P}{\partial y}$ has degree at most $d$.
Since elements of $W$ have degree at most $d - 1$, the initial state is not necessarily an element of $W$.
However, the following result shows that most of its images under compositions of $\lambda_{r, 0}$ are elements of $W$.
A similar result was obtained in \cite[Remark~2.6]{Bostan--Caruso--Christol--Dumas}.

\begin{proposition}\label{transient}
Let $S \in \field[x, y]$ such that $\deg_x S \leq h$ and $\deg_y S \leq d$.
\begin{itemize}
\item We have $\deg_x \lambda_{0, 0}(S) \leq h$ and $\deg_y \lambda_{0, 0}(S) \leq d - 1$.
\item For each $r \in \{1, \dots, q - 1\}$, we have $\lambda_{r, 0}(S) \in W$.
\end{itemize}
In particular, every polynomial $(\lambda_{r_n, 0} \circ \dots \circ \lambda_{r_2, 0} \circ \lambda_{r_1, 0})(S_0)$, where at least one $r_i$ is not $0$, is an element of $W$.
\end{proposition}

\begin{proof}
For the first statement, we follow the proof of Proposition~\ref{module}.
After setting $r = 0$, Equation~\eqref{image height} is replaced with
\[
	\tfrac{i + I}{q}
	= \floor{\tfrac{i + I}{q}}
	\le \floor{\tfrac{h + (q - 1) h}{q}}
	= \floor{\tfrac{q h}{q}}
	= h.
\]
Therefore $\deg_x S \leq h$.
Similarly, in the case that $\Lambda_{0, 0}\!\paren{x^i y^j \cdot x^I y^J}$ is not $0$, its degree satisfies
\[
	\tfrac{j + J}{q}
	= \floor{\tfrac{j + J}{q}}
	\leq \floor{\tfrac{d + (q - 1) (d - 1)}{q}}
	= \floor{\tfrac{q (d - 1) + 1}{q}}
	= d - 1.
\]

For the second statement, we also follow the proof of Proposition~\ref{module};
Equation~\eqref{image height} is replaced with
\[
	\tfrac{i + I - r}{q}
	= \floor{\tfrac{i + I - r}{q}}
	\le \floor{\tfrac{h + (q - 1) h - r}{q}}
	= \floor{\tfrac{q h - r}{q}}
	\leq h - 1,
\]
and analogously
\[
	\tfrac{j + J}{q}
	= \floor{\tfrac{j + J}{q}}
	\leq \floor{\tfrac{d + (q - 1) (d - 1) - r}{q}}
	= \floor{\tfrac{q (d - 1) + 1 - r}{q}}
	\leq d - 1.
\]

The initial state $S_0 = y \frac{\partial P}{\partial y}$ satisfies $\deg_x S_0 \leq h$ and $\deg_y S_0 \leq d$.
Let $r \in \{1, \dots, q - 1\}$.
Applying both statements, we obtain $(\lambda_{r, 0} \circ \lambda_{0, 0}^n)(S_0) \in W$ for all $n \geq 1$.
Therefore, by Proposition~\ref{module}, the final statement follows.
\end{proof}

An immediate corollary of Propositions~\ref{module} and \ref{transient} is the following, since all states $S$ except the initial state satisfy $\deg_x S\leq h$ and $\deg_y S\leq d-1$.

\begin{corollary}\label{kernel initial upper bound}
Let $F = \sum_{n \geq 0} a(n) x^n \in \field\doublebracket{x}$ be the Furstenberg series associated with a polynomial $P \in \field[x, y]$ of height $h$ and degree $d$.
Then
\[
	\size{\ker_q(a(n)_{n\geq 0})}
	\leq q^{(h+1) d} + 1.
\]
\end{corollary}

Proposition~\ref{transient} indicates that we must further study $\lambda_{0, 0}$ to lower the bound in Corollary~\ref{kernel initial upper bound}.
We start to do this next.
For a function $f \colon X \to X$, define the \emph{orbit} of $S \in X$ under $f$ to be the sequence $S, f(S), f^2(S), \dots$, and let $\size{\orb_f(S)}$ be the number of distinct terms in the orbit.

\begin{corollary}\label{kernel preliminary upper bound}
Let $F = \sum_{n \geq 0} a(n) x^n \in \field\doublebracket{x}$ be the Furstenberg series associated with a polynomial $P \in \field[x, y]$ of height $h$ and degree $d$.
Then
\[
	\size{\ker_q(a(n)_{n\geq 0})}
	\leq q^{h d} + \size{\orb_{\Lambda_0}(F)}.
\]
\end{corollary}

\begin{proof}
Recall that each sequence in $\ker_q(a(n)_{n\geq 0})$ is represented by at least one polynomial obtained by iteratively applying some sequence of the operators $\lambda_{r, 0}$ to the initial state $S_0 = y\tfrac{\partial P}{\partial y}$.
Applying $\lambda_{0, 0}$ iteratively to $S_0$ produces $\size{\orb_{\lambda_{0, 0}}(S_0)}$ states, and $\size{\orb_{\lambda_{0, 0}}(S_0)} = \size{\orb_{\Lambda_0}(F)}$ by definition.
By Propositions~\ref{module} and \ref{transient}, all states that are not in $\orb_{\lambda_{0, 0}}(S_0)$ are in $W$, which has size $q^{h d}$.
\end{proof}

\section{Structure of the linear transformation $\lambda_{0, 0}$}\label{section: structure}

By Corollary~\ref{kernel preliminary upper bound}, it remains to bound $\size{\orb_{\Lambda_0}(F)}$.
In this section, we take the first step toward this goal by identifying univariate operators $\lambda_0$ that emulate $\lambda_{0, 0}$ on three subspaces.
The main result is Proposition~\ref{univariate emulation}.

We continue to use the notation $h$, $d$, $P$, $Q$, $W$ established in the previous section.
As we saw with regard to Proposition~\ref{transient}, the elements of $\orb_{\lambda_{0, 0}}(S_0)$ do not necessarily belong to $W$.
However, they do belong to the slightly larger space
\[
	V \colonequal \left\langle x^i y^j : \text{$0 \leq i \leq h$ and $0 \leq j \leq d - 1$}\right\rangle.
\]
We define three subspaces of $V$, which we label suggestively using $\ell$~(left), r~(right), and t~(top):
\begin{align*}
	V\subl &= \left\langle x^0 y^j : 0 \leq j \leq d - 1 \right\rangle \\
	V\subr &= \left\langle x^h y^j : 0 \leq j \leq d - 1 \right\rangle \\
	V\subt &= \left\langle x^i y^{d - 1} : 0 \leq i \leq h \right\rangle.
\end{align*}
We also define the \emph{interior} of $V$ to be
\begin{equation}\label{V interior definition}
	V^\circ = \left\langle x^i y^j : \text{$1 \leq i \leq h - 1$ and $0 \leq j \leq d - 2$}\right\rangle.
\end{equation}
Note that, despite the name ``interior'', the basis of $V^\circ$ contains monomials $x^i y^0$ along the bottom edge of the rectangle.
We have $V^\circ \cap V\subl = V^\circ \cap V\subr = V^\circ \cap V\subt = \{0\}$.
We will see that the factor $q^{(h - 1) (d - 1)}$ in Theorem~\ref{kernel size upper bound} comes from the size of $V^\circ$.

To establish the structure of $\lambda_{0, 0}$, we introduce three projection-like maps.

\begin{notation*}
Let $\pi\subl \colon \field[x,y] \to \field[y]$ denote the projection map from $\field[x,y]$ to $V\subl$.
We define $\pi\subr$ slightly differently.
We have $V\subr \subset x^h \field[y]$, so rather than projecting to $V\subr$ we will dispense with the factor $x^h$.
Namely, define $\pi\subr \colon \field[x,y] \to \field[y]$ by $\pi\subr(S) = \frac{1}{x^h} \rho(S)$, where $\rho$ projects from $\field[x,y]$ to $V\subr$.
Similarly, define $\pi\subt \colon \field[x,y] \to \field[x]$ by $\pi\subt(S) = \frac{1}{y^{d - 1}} \rho(S)$, where $\rho$ projects from $\field[x,y]$ to $V\subt$.
\end{notation*}

\begin{example}\label{example projections}
As in Example~\ref{example}, let $q = 3$ and
\begin{align*}
	Q &= (x^2 + x + 2) y^3 + x y^2 + (2 x + 1) y + x^2 + 1 + (2 x^2 + x) y^{-1} \\
	S_0
	&= (x^2 + x + 2) y^4 + (x + 2) y^2 + (x^2 + 1) y.
\end{align*}
The second state in the orbit of $S_0$ under $\lambda_{0, 0}$ is
\[
	S_1
	\colonequal \lambda_{0, 0}(S_0)
	= \Lambda_{0, 0}(S_0 Q^{3 - 1})
	= x y^3 + (x^2 + x + 1) y^2 + (2 x^2 + 2) y + x^2 + x.
\]
We have $S_1 \in V$, which is consistent with Proposition~\ref{transient}.
Since $h = 2$ and $d = 4$, the images of $S_1$ under $\pi\subl, \pi\subr, \pi\subt$ are
\begin{align*}
	\pi\subl(S_1) &= y^2 + 2 y \\
	\pi\subr(S_1) &= y^2 + 2 y + 1 \\
	\pi\subt(S_1) &= x.
\end{align*}
We use these projections in Example~\ref{example denominator projections} below.
\end{example}

Next we define univariate versions of $\lambda_{0,0}$.
We will use the symbol $z$ to denote either $x$ or $y$, depending on which subspace we are considering.

\begin{notation*}
Let $R \in z^{-1} \field[z]$.
Define $\lambda_0 \colon \field[z] \to \field[z]$ by
\begin{equation}\label{univariate little lambda definition}
	\lambda_0(S) = \Lambda_0\!\paren{S R^{q - 1}}.
\end{equation}
\end{notation*}

The next proposition shows that $\lambda_0$ emulates $\lambda_{0, 0}$ on the subspaces $V\subl$, $V\subr$, and $V\subt$.
Each of the three statements describes a commuting diagram.
For example, the first statement says that the diagram
\begin{center}\begin{tikzcd}
	\field[x,y] \arrow[r, "\lambda_{0,0}"]\arrow[d,"\pi\subl"'] & \field[x,y]\arrow[d,"\pi\subl"] \\
	\field[y] \arrow[r,"\lambda_0"] & \field[y]
\end{tikzcd}\end{center}
commutes.
Write
\begin{equation}\label{equation: P form}
	P(x,y) = \sum_{i = 0}^h x^i A_i(y) = \sum_{j = 0}^d B_j(x) y^j.
\end{equation}
Note that $A_0/y$ is a polynomial since we assume $P(0,0) = 0$ for a Furstenberg series.

\begin{proposition}\label{univariate emulation}
We have the following.

\begin{enumerate}
\item\label{left border}
Let $R = A_0/y$.
For all $S \in \field[x,y]$,
\[
	\pi\subl(\lambda_{0, 0}(S))
	= \lambda_0(\pi\subl(S)).
\]
\item\label{right border}
Let $R = A_h/y$.
For all $S \in \field[x,y]$ with height at most $h$,
\[
	\pi\subr(\lambda_{0, 0}(S))
	= \lambda_0(\pi\subr(S)).
\]
In particular, $\lambda_0(\pi\subr(S))$ is a polynomial despite $R$ not necessarily being a polynomial.
\item\label{top border}
Let $R = B_d$.
For all $S \in \field[x,y]$ with degree at most $d - 1$,
\[
	\pi\subt(\lambda_{0, 0}(S))
	= \lambda_0(\pi\subt(S)).
\]
\end{enumerate}
\end{proposition}

In particular, Proposition~\ref{univariate emulation} implies that the $V\subl$, $V\subr$, and $V\subt$ components of $\lambda_{0, 0}(S)$ depend only on the respective $V\subl$, $V\subr$, and $V\subt$ components of $S$.

\begin{example}\label{example denominator projections}
For the polynomial $P$ in Example~\ref{example}, we have
\begin{align*}
	A_0/y &= 2 y^3 + y + 1 \\
	A_h/y &= y^3 + 1 + 2y^{-1} \\
	B_d &= x^2 + x + 2.
\end{align*}
With these respective values of $R$, the second state $S_1$ in the orbit of $S_0$ under $\lambda_{0, 0}$, computed in Example~\ref{example projections}, satisfies
\begin{align*}
	\pi\subl(\lambda_{0, 0}(S_1)) &= y^2 + y = \lambda_0(\pi\subl(S_1)) \\
	\pi\subr(\lambda_{0, 0}(S_1)) &= y^2 + y + 1 = \lambda_0(\pi\subr(S_1)) \\
	\pi\subt(\lambda_{0, 0}(S_1)) &= 2 x = \lambda_0(\pi\subt(S_1)),
\end{align*}
confirming the statement of Proposition~\ref{univariate emulation}.
That is, Proposition~\ref{univariate emulation} reduces the computation of $\pi\subl(\lambda_{0, 0}(S_1))$ to the univariate computation of $\lambda_0(\pi\subl(S_1))$, and similarly for $\pi\subr$ and $\pi\subt$.
\end{example}

\begin{proof}[Proof of Proposition~\ref{univariate emulation}]
First we consider $\pi\subl(\lambda_{0, 0}(S))$ for $S \in \field[x,y]$.
Since $\pi\subl$ projects onto polynomials in $y$, we are interested in monomials with height~$0$ in $\lambda_{0, 0}(S) = \Lambda_{0, 0}\!\paren{S Q^{q - 1}}$.
A monomial $c \, x^0 y^J$ in $SQ^{q - 1}$ arises only from the product of a monomial in $S$ with height $0$ together with a monomial in $Q^{q - 1}$ with height~$0$, that is, only from the product of a monomial in $\pi\subl(S)$ together with a monomial in $Q^{q - 1}$ with height~$0$.
Therefore
\[
	\pi\subl(\lambda_{0, 0}(S))
	= \pi\subl(\lambda_{0, 0}(\pi\subl(S))).
\]
Additionally, the only way to get a monomial in $Q^{q - 1}$ with height~$0$ is to take a product of $q - 1$ monomials in $Q = P/y$ with height~$0$, namely, monomials in $A_0/y$.
Therefore,
\[
	\pi\subl(\lambda_{0, 0}(S))
	= \pi\subl\!\left(\Lambda_{0, 0}\!\paren{\pi\subl(S) \cdot (A_0/y)^{q - 1}}\right).
\]
Since $\pi\subl(S) \cdot (A_0/y)^{q - 1}$ is a univariate polynomial in $y$, we obtain
\[
	\pi\subl(\lambda_{0, 0}(S))
	= \Lambda_0\!\paren{\pi\subl(S) (A_0/y)^{q - 1}}
	= \lambda_0(\pi\subl(S)).
\]

The argument is similar for $\pi\subr(\lambda_{0, 0}(S))$.
Let $\deg_x S \leq h$.
We have $\pi\subr(x^I y^J) = 0$ if $I \neq h$.
Since $\deg_{x} Q = h$, each monomial $c \, x^{q h} y^J$ in each of $S\cdot Q^{q - 1}$ and $x^h \pi\subr(S) \cdot Q^{q - 1}$ arises only from the product of a monomial in $x^h \pi\subr(S)$ together with a product of $q - 1$ monomials in $Q$ with height~$h$, namely, monomials in $x^h A_h/y$.
Therefore
\begin{align*}
	\pi\subr(\lambda_{0, 0}(S))
	= \pi\subr(\lambda_{0, 0}(x^h \pi\subr(S)))
	&= \pi\subr\!\left(\Lambda_{0, 0}\!\paren{x^h \pi\subr(S) \cdot (x^h A_h/y)^{q - 1}}\right) \\
	&= \pi\subr\!\left(x^h \Lambda_{0, 0}\!\paren{\pi\subr(S) (A_h/y)^{q - 1}}\right) \\
	&= \Lambda_0\!\paren{\pi\subr(S) (A_h/y)^{q - 1}} \\
	&= \lambda_0(\pi\subr(S)),
\end{align*}
where in the third equality we use Proposition~\ref{Cartier} to rewrite $\Lambda_{0, 0}(G x^{h q}) = x^h \Lambda_{0, 0}(G)$.
Moreover, $\lambda_0(\pi\subr(S))$ is a polynomial since monomials in $\pi\subr(S)$ have degree at least $0$ and monomials in $(A_h/y)^{q - 1}$ have degree at least $-(q - 1)$.

Finally, we consider $\pi\subt(\lambda_{0, 0}(S))$ for $\deg_y S \leq d - 1$.
We have $\pi\subt(x^I y^J) = 0$ if $J \neq d - 1$.
Since $\deg_{y} Q = d - 1$, each monomial $c \, x^I y^{q (d - 1)}$ in each of $S\cdot Q^{q - 1}$ and $\pi\subt(S) y^{d - 1} \cdot Q^{q - 1}$ arises only from the product of a monomial in $\pi\subt(S) y^{d - 1}$ together with a product of $q - 1$ monomials in $Q$ with degree~$d - 1$, namely, monomials in $B_d y^{d - 1}$.
Therefore
\begin{align*}
	\pi\subt(\lambda_{0, 0}(S))
	= \pi\subt(\lambda_{0, 0}(\pi\subt(S) y^{d - 1}))
	&= \pi\subt\!\left(\Lambda_{0, 0}\!\paren{\pi\subt(S) y^{d - 1} \cdot ( B_d y^{d - 1})^{q - 1}}\right) \\
	&= \pi\subt\!\left(y^{d - 1} \Lambda_{0, 0}\!\paren{\pi\subt(S) B_d^{q - 1}}\right) \\
	&= \Lambda_0\!\paren{\pi\subt(S) B_d^{q - 1}} \\
	&= \lambda_0(\pi\subt(S)).
	\qedhere
\end{align*}
\end{proof}

\subsection{The linear structure of $\lambda_{0,0}$}

Proposition~\ref{univariate emulation} identifies three subspaces on which $\lambda_{0, 0}$ is equivalent to a univariate operator $\lambda_0$.
This proposition is sufficient for the proof of Theorem~\ref{kernel size upper bound}, which we resume in Section~\ref{section: orbit size univariate}.
However, in the remainder of this section we develop additional intuition by using Proposition~\ref{univariate emulation} to refine, in two steps, the standard basis of $V$ to reveal additional structure of the linear transformation $\lambda_{0, 0}$ and its corresponding matrix.

Define
\begin{align*}
	V\subl^\circ &= \left\langle x^0 y^j : 1 \leq j \leq d - 2 \right\rangle \\
	V\subr^\circ &= \left\langle x^h y^j : 0 \leq j \leq d - 2 \right\rangle \\
	V\subt^\circ &= \left\langle x^i y^{d - 1} : 1 \leq i \leq h - 1 \right\rangle
\end{align*}
so that
\begin{align*}
	V\subl &=
		\left\langle x^0 y^0 \right\rangle
		\oplus V\subl^\circ
		\oplus \left\langle x^0 y^{d - 1} \right\rangle \\
	V\subr &=
		V\subr^\circ
		\oplus \left\langle x^h y^{d - 1} \right\rangle \\
	V\subt &=
		\left\langle x^0 y^{d - 1} \right\rangle
		\oplus V\subt^\circ
		\oplus \left\langle x^h y^{d - 1} \right\rangle.
\end{align*}

\begin{figure}
	\begin{tikzpicture}
	\draw [draw = black] (0,0) rectangle (5,5) ;

	\draw [draw = gray, dashed] (1,0) -- (1,5);
	\draw [draw = gray, dashed] (4,0) -- (4,5);
	\draw [draw = gray, dashed] (0,4) -- (5,4);
	\draw [draw = gray, dashed] (0,1) -- (1,1);

	\node[scale = 0.8] at (0.5,0.5) {$x^0 y^0$};
	\node[scale = 0.8] at (1.5,0.5) {$x^1 y^0$};
	\node[scale = 0.8] at (2.4,0.5) {\textcolor{gray}{$\ldots$}};
	\node[scale = 0.8] at (3.4,0.5) {$x^{h - 1} y^0$};
	\node[scale = 0.8] at (4.5,0.5) {$x^h y^0$};

	\node[scale = 0.8] at (0.5,1.5) {$x^0 y^1$};
	\node[scale = 0.8] at (1.5,1.5) {$x^1 y^1$};
	\node[scale = 0.8] at (2.4,1.5) {\textcolor{gray}{$\ldots$}};
	\node[scale = 0.8] at (3.4,1.5) {$x^{h - 1} y^1$};
	\node[scale = 0.8] at (4.5,1.5) {$x^h y^1$};

	\node[scale = 0.8] at (0.5,2.6) {\textcolor{gray}{$\vdots$}};
	\node[scale = 0.8] at (1.5,2.6) {\textcolor{gray}{$\vdots$}};
	\node[scale = 0.8] at (2.4,2.6) {\textcolor{gray}{$\iddots$}};
	\node[scale = 0.8] at (3.4,2.6) {\textcolor{gray}{$\vdots$}};
	\node[scale = 0.8] at (4.5,2.6) {\textcolor{gray}{$\vdots$}};

	\node[scale = 0.8] at (0.5,3.5) {$x^0 y^{d-2}$};
	\node[scale = 0.8] at (1.5,3.5) {$x^1 y^{d-2}$};
	\node[scale = 0.8] at (2.4,3.5) {\textcolor{gray}{$\ldots$}};
	\node[scale = 0.8] at (3.4,3.5) {$x^{h - 1} y^{d-2}$};
	\node[scale = 0.8] at (4.5,3.5) {$x^h y^{d-2}$};

	\node[scale = 0.8] at (0.5,4.5) {$x^0 y^{d - 1}$};
	\node[scale = 0.8] at (1.5,4.5) {$x^1 y^{d - 1}$};
	\node[scale = 0.8] at (2.4,4.5) {\textcolor{gray}{$\ldots$}};
	\node[scale = 0.8] at (3.4,4.5) {$x^{h - 1} y^{d - 1}$};
	\node[scale = 0.8] at (4.5,4.5) {$x^h y^{d - 1}$};
	\end{tikzpicture}
	\caption{Partition of the basis of $V$ into seven sets, which generate the subspaces $\left\langle x^0 y^{d - 1} \right\rangle$, $V\subt^\circ$, $\left\langle x^h y^{d - 1} \right\rangle$, $V\subl^\circ$, $V^\circ$, $V\subr^\circ$, and $\left\langle x^0 y^0 \right\rangle$.}
		\label{seven subspaces}
\end{figure}

The bases of the seven subspaces
\begin{equation}\label{subspace order}
	V^\circ,
	\quad V\subl^\circ,
	\quad \left\langle x^0 y^0 \right\rangle,
	\quad V\subt^\circ,
	\quad \left\langle x^0 y^{d - 1} \right\rangle,
	\quad V\subr^\circ,
	\quad \left\langle x^h y^{d - 1} \right\rangle
\end{equation}
are disjoint and form a set partition of the basis of $V$.
Geometrically, these bases are arranged as in Figure~\ref{seven subspaces}.
We will show in Corollary~\ref{block upper triangular} that, with this decomposition of $V$, the matrix corresponding to $\lambda_{0, 0}$ is block upper triangular.
The block sizes are $
	(h - 1) (d - 1),
	d - 2,
	1,
	h - 1,
	1,
	d - 1,
	1$.

\begin{example}
As in Example~\ref{example denominator projections}, let $h = 2$, $d = 4$, and
\[
	P = (x^2 + x + 2) y^4 + x y^3 + (2 x + 1) y^2 + (x^2 + 1) y + 2 x^2 + x \in \F_3[x, y].
\]
The basis of $V$, ordered according to~\eqref{subspace order}, is
\[
	\left(
	x^1 y^0, x^1 y^1, x^1 y^2, \quad
	x^0 y^1, x^0 y^2, \quad
	x^0 y^0, \quad
	x^1 y^3, \quad
	x^0 y^3, \quad
	x^2 y^0, x^2 y^1, x^2 y^2, \quad
	x^2 y^3
	\right).
\]
With this basis, the operators $\lambda_{0, 0}$, $\lambda_{1, 0}$, and $\lambda_{2, 0}$ are represented by the $12 \times 12$ matrices
\tiny
\begin{align*}
\setcounter{MaxMatrixCols}{12}
	L_{0, 0} &=
	\begin{bmatrix}
		1 & 1 & 1 & 2 & 1 & 2 & 0 & 0 & 2 & 2 & 0 & 0 \\
		1 & 2 & 1 & 2 & 2 & 2 & 1 & 2 & 1 & 1 & 1 & 2 \\
		2 & 2 & 1 & 2 & 1 & 2 & 1 & 2 & 1 & 1 & 1 & 1 \\
		& & & 1 & 2 & 1 & & 1 & & & & \\
		& & & 0 & 1 & 1 & & 1 & & & & \\
		& & & & & 1 & & & & & & \\
		& & & & & & 2 & 2 & & & & 1 \\
		& & & & & & & 1 & & & & \\
		& & & & & & & & 1 & 1 & 1 & 0 \\
		& & & & & & & & 2 & 1 & 0 & 1 \\
		& & & & & & & & 1 & 0 & 0 & 2 \\
		& & & & & & & & & & & 1
	\end{bmatrix} \\
	L_{1, 0} &=
	\begin{bmatrix}
		2 & 2 & 1 & 1 & 1 & 1 & 0 & 0 & 1 & 1 & 1 & 0 \\
		2 & 2 & 2 & 1 & 0 & 2 & 2 & 1 & 1 & 2 & 1 & 1 \\
		2 & 2 & 1 & 0 & 0 & 1 & 2 & 2 & 2 & 2 & 1 & 1 \\
		1 & 1 & 2 & 1 & 1 & 1 & 1 & 2 & 0 & 0 & 0 & 0 \\
		1 & 0 & 1 & 1 & 1 & 1 & 1 & 1 & 0 & 0 & 0 & 0 \\
		1 & 0 & 0 & 2 & 0 & 2 & 0 & 0 & 0 & 0 & 0 & 0 \\
		& & & & & & 2 & 1 & 0 & 0 & 0 & 2 \\
		& & & & & & 1 & 1 & 0 & 0 & 0 & 0 \\
		& & & & & & & 0 & 0 & 0 & 0 & 0 \\
		& & & & & & & & 0 & 0 & 0 & 0 \\
		& & & & & & & & 0 & 0 & 0 & 0 \\
		& & & & & & & & & & & 0
	\end{bmatrix} \\
	L_{2, 0} &=
	\begin{bmatrix}
		1 & 1 & 1 & 0 & 0 & 0 & 0 & 0 & 2 & 2 & 1 & 0 \\
		2 & 1 & 0 & 0 & 0 & 0 & 1 & 0 & 2 & 2 & 2 & 2 \\
		1 & 0 & 0 & 0 & 0 & 0 & 2 & 0 & 2 & 2 & 1 & 2 \\
		1 & 1 & 1 & 2 & 1 & 1 & 2 & 1 & 1 & 1 & 2 & 1 \\
		1 & 1 & 1 & 2 & 1 & 2 & 1 & 1 & 1 & 0 & 1 & 1 \\
		2 & 2 & 0 & 1 & 1 & 1 & 0 & 0 & 1 & 0 & 0 & 0 \\
		& & & & & & 1 & 0 & 0 & 0 & 0 & 2 \\
		& & & & & & 1 & 2 & 0 & 0 & 0 & 1 \\
		& & & & & & & 0 & 0 & 0 & 0 & 0 \\
		& & & & & & & & 0 & 0 & 0 & 0 \\
		& & & & & & & & 0 & 0 & 0 & 0 \\
		& & & & & & & & & & & 0
	\end{bmatrix}.
\end{align*}
\normalsize
The first three columns of $L_{0, 0}$ have $0$s in rows~$4$--$12$, since Proposition~\ref{univariate emulation} tells us that the $V^\circ$ component of $S$ has no impact on the $V\subl$, $V\subr$, and $V\subt$ components of $\lambda_{0, 0}(S)$.
Conversely, several nonzero entries in the last column of $L_{0, 0}$ indicate that the monomial $x^2 y^3 \in V\subr$ in $S$ has an effect on monomials of $\lambda_{0, 0}(S)$ outside of $V\subr$.
In general, entries guaranteed to be $0$ by Theorem~\ref{information flow} below have been omitted from $L_{0, 0}$, and entries guaranteed to be $0$ by Corollary~\ref{block upper triangular} have been omitted from $L_{1, 0}$ and $L_{2, 0}$.
In addition, the second statement in Proposition~\ref{transient} implies that the last $d = 4$ rows of $L_{1, 0}$ and $L_{2, 0}$ are zero rows.
\end{example}

Proposition~\ref{univariate emulation} shows that, under applications of $\lambda_{0, 0}$, information flows from $V\subl$ to its complement subspace but not in the other direction, and similarly for $V\subr$ and $V\subt$.
That is, information flows between four subspaces of $V$ according to the following diagram.
\begin{center}
	\begin{tikzcd}
		\left\langle x^0 y^{d - 1} \right\rangle \oplus V\subt^\circ
			\arrow[gray, out = 60, in = 120, loop, looseness = 5]
		\arrow[gray, d]
		& \left\langle x^h y^{d - 1} \right\rangle
			\arrow[gray, out = 60, in = 120, loop, looseness = 5]
			\arrow[gray, l]
			\arrow[gray, d]
			\arrow[gray, dl]
		\\
		\left\langle x^0 y^0 \right\rangle \oplus V\subl^\circ \oplus V^\circ
			\arrow[gray, out = 240, in = 300, loop, looseness = 4.5]
		& V\subr^\circ
			\arrow[gray, out = 240, in = 300, loop, looseness = 5]
			\arrow[gray, l]
	\end{tikzcd}
\end{center}

We can refine this further.

\begin{theorem}\label{information flow}
Under applications of $\lambda_{0, 0}$ on $V$, information flows according to the following diagram.
Namely, if $S \in V$ and $U$ is one of the seven distinguished subspaces of $V$, then the projection of $\lambda_{0, 0}(S)$ to $U$ is determined by the projections of $S$ onto the subspaces with arrows pointing to $U$.
\begin{center}
	\begin{tikzcd}
		\left\langle x^0 y^{d - 1} \right\rangle
			\arrow[gray, out = 105, in = 165, loop, looseness = 3]
			\arrow[gray, r]
			\arrow[gray, d]
			\arrow[gray, dr]
		& V\subt^\circ
			\arrow[gray, out = 60, in = 120, loop, looseness = 5]
			\arrow[gray, d]
		& \left\langle x^h y^{d - 1} \right\rangle
			\arrow[gray, out = 15, in = 75, loop, looseness = 3]
			\arrow[gray, l]
			\arrow[gray, d]
			\arrow[gray, dl]
		\\
		V\subl^\circ
			\arrow[gray, out = 150, in = 210, loop, looseness = 4]
			\arrow[gray, r]
		& V^\circ
			\arrow[gray, out = 240, in = 300, loop, looseness = 6]
		& V\subr^\circ
			\arrow[gray, out = 330, in = 30, loop, looseness = 4]
			\arrow[gray, l]
		\\
		\left\langle x^0 y^0 \right\rangle
			\arrow[gray, out = 195, in = 255, loop, looseness = 3]
			\arrow[gray, u]
			\arrow[gray, ur]
	\end{tikzcd}
\end{center}
\end{theorem}

\begin{proof}
We will rule out all arrows that do not appear in the diagram.

Part~\ref{left border} of Proposition~\ref{univariate emulation} implies that the left subspaces $\left\langle x^0 y^0 \right\rangle$, $V\subl^\circ$, and $\left\langle x^0 y^{d - 1} \right\rangle$ have no incoming arrows from the other four subspaces.
Similarly, Part~\ref{right border} implies that the right subspaces $V\subr^\circ$ and $\left\langle x^h y^{d - 1} \right\rangle$ have no incoming arrows from the other five subspaces, and Part~\ref{top border} implies that the top subspaces $\left\langle x^0 y^{d - 1} \right\rangle$, $V\subt^\circ$, and $\left\langle x^h y^{d - 1} \right\rangle$ have no incoming arrows from the other four subspaces.
It follows that the top corner subspaces $\left\langle x^0 y^{d - 1} \right\rangle$ and $\left\langle x^h y^{d - 1} \right\rangle$ have no incoming arrows other than their loops.

To see that $\left\langle x^0 y^0 \right\rangle$ has no incoming arrows other than its loop, let $j \in \{1, \dots, d - 1\}$.
We have $\lambda_{0, 0}(x^0 y^j) = \Lambda_{0, 0}(x^0 y^j Q^{q - 1})$.
The coefficient of $x^0 y^0$ in $\lambda_{0, 0}(x^0 y^j)$ is equal to the coefficient of $x^0 y^0$ in $x^0 y^j Q^{q - 1}$.
However, since $P(0, 0) = 0$ and $Q = P / y$, the only monomials $x^I y^J$ with $J \leq -1$ that appear in $Q$ with a nonzero coefficient satisfy $I \geq 1$.
Therefore the coefficient of $x^0 y^0$ in $x^0 y^j Q^{q - 1}$ is $0$, and $\left\langle x^0 y^0 \right\rangle$ has only one incoming arrow.
\end{proof}

Theorem~\ref{information flow} and Proposition~\ref{transient} imply the following, where the seven blocks correspond to the seven subspaces in Theorem~\ref{information flow}.

\begin{corollary}\label{block upper triangular}
If the basis of $V$ is ordered according to~\eqref{subspace order}, then the matrix corresponding to $\lambda_{0, 0}$ is block upper triangular with seven blocks.
Moreover, for all $r \in\{1,2,\dots,q - 1\}$, the matrix corresponding to $\lambda_{r, 0}$ is block upper triangular with four blocks (whose sizes are $h (d - 1)$, $h$, $d - 1$, and $1$).
\end{corollary}

\section{Orbit size of a univariate polynomial under $\lambda_0$}\label{section: orbit size univariate}

Proposition~\ref{univariate emulation} (and, more explicitly, Theorem~\ref{information flow}) shows that the orbit size of a bivariate polynomial $S \in V$ under $\lambda_{0, 0}$ depends in part on the orbit sizes of the univariate polynomials $\pi\subl(S), \pi\subr(S), \pi\subt(S)$ under $\lambda_0$ for the respective values $R = A_0/y, R = A_h/y, R = B_d$.
(Recall from Equation~\eqref{univariate little lambda definition} that the definition of $\lambda_0$ depends on $R$.)
The main result of this section is Theorem~\ref{univariate orbit size upper bound}, which establishes an upper bound on orbit sizes under $\lambda_0$ for a general element $R \in z^{-1} \field[z]$; this includes the case $R = A_h/y$ (where $z = y$), which is not necessarily a polynomial (for instance, as in Example~\ref{example denominator projections}).

We will use the following lemma several times.

\begin{lemma}\label{iterating floor}
Let $q \geq 2$.
\begin{itemize}
\item
If $k \in \Z$ and $f(x) = \floor{\frac{x + k (q - 1)}{q}}$, then, for every $x \geq k$ and $n \geq \floor{\log_q(x - k)} + 1$, we have $f^n(x) = k$.
\item
If $k\geq 1$ and $f(x) = \ceil{\frac{x + k (q - 1)}{q}}$, then, for every $x\geq 0$ and $n \geq \floor{\log_q k} + 1$, we have $f^n(x) \geq k$.
\end{itemize}
\end{lemma}

\begin{proof}
The function $f(x) = \floor{\frac{x + k (q - 1)}{q}}= k + \floor{\frac{x - k}{q}}$ has an attracting fixed point $k$ for $x \geq k$.
Since $\floor{\frac{\floor{(x - k)/q^n}}{q}} = \floor{\frac{x - k}{q^{n + 1}}}$, a straightforward induction shows that $f^n(x) = k + \floor{\frac{x - k}{q^n}}$ for all $n \geq 0$.
The first statement follows.

For the second statement, we have $f^n(x) = k + \ceil{\frac{x - k}{q^n}}$ for all $n \geq 0$.
If $n \geq \floor{\log_q k} + 1$, then $\ceil{-\frac{k}{q^n}} = 0$.
Therefore $f^n(x) = k + \ceil{\frac{x - k}{q^n}} \geq k + \ceil{-\frac{k}{q^n}} = k$.
\end{proof}

The following proposition shows that if $\deg S > \deg R$ then the orbit of $S$ under $\lambda_0$ eventually consists of polynomials with degree at most $\deg R$.
Looking ahead, this will let us restrict attention to polynomials $S$ with $\deg S \leq \deg R$ in later results (namely, Theorem~\ref{square-free orbit size upper bound - positive degree}, Corollary~\ref{square-free orbit size upper bound}, and Theorem~\ref{univariate orbit size upper bound}).
We will use it directly in the proof of Lemma~\ref{right degree}.

\begin{proposition}\label{fixed point}
Let $R \in z^{-1} \field[z]$ be a Laurent polynomial, let $r = \deg R$, and define $\lambda_0$ on $\field[z]$ by $\lambda_0(S) = \Lambda_0\!\paren{S R^{q - 1}}$.
Let $S \in \field[z]$, let $s = \deg S$, and suppose that $s>r$.
If $n\geq \floor{\log_q(s - r)} + 1$, then $\deg \lambda_0^n(S)\leq r$.
\end{proposition}

In particular, if $r = -1$ then $\lambda_0^n(S) = 0$ for sufficiently large $n$ since $\lambda_0$ maps polynomials to polynomials.

\begin{proof}
We have $\deg \lambda_0(S) \leq \frac{s + r (q - 1)}{q}$.
We track the behavior of $\deg \lambda_0^n(S)$ by iterating the function $f(x) = \floor{\frac{x + r (q - 1)}{q}}$.
Applying Lemma~\ref{iterating floor} with $k = r$, the result follows.
\end{proof}

\begin{example}
Let $q = 3$ and $R = (z^2 + 1) (z^3 + z^2 + 2) \in \F_3[z]$.
By computing $\orb_{\lambda_0}(S)$ from each $S \in \F_3[z]$ with $\deg S \leq \deg R = 5$, one finds that each orbit is periodic with period length $1$, $2$, $3$, or $6$.
For example, the orbit of $z^4 + z^2$ has period length $1$, the orbit of $z^2 + 2 z + 1$ has period length $2$, the orbit of $z + 2$ has period length $3$, and the orbit of $1$ has period length $6$.
\end{example}

The orbit of $S \in \field[z]$ under $\lambda_0$ is eventually periodic, since an argument similar to Proposition~\ref{fixed point} shows that the elements in the orbit have bounded degree.
As we vary $q$ and $r$ and consider all polynomials $R \in \field[z]$ with fixed degree $\deg R = r$, one finds that the maximal size of the orbit is independent of $q$ and depends only on $r$.
We prove this in Theorem~\ref{square-free orbit size upper bound - positive degree}, which is an important step in proving Theorem~\ref{univariate orbit size upper bound}.
The proof uses the periodicity of the series expansion of $\frac{1}{R}$ to establish the periodicity of the orbit under $\lambda_0$, as in the following example.

\begin{example}
Let $q = 2$ and $R = z^2 + z + 1 \in \F_2[z]$.
In light of Proposition~\ref{fixed point}, we consider polynomials $S \in \F_2[z]$ such that $\deg S \leq \deg R = 2$.
Let $j \in \{0, 1, 2\}$, so that each monomial in $S$ is of the form $c \, z^j$.
Proposition~\ref{Cartier} implies $\lambda_0(z^j) = \Lambda_0(z^j R^{q - 1}) = \Lambda_0(\frac{z^j}{R}) R$.
Iterating $\lambda_0$ gives $\lambda_0^n(z^j) = \Lambda_0^n(\frac{z^j}{R}) R$ for all $n \geq 0$.
We show that $\Lambda_0^2(\frac{z^j}{R}) = \frac{z^j}{R}$; this implies $\lambda_0^2(z^j) = z^j$, which, by linearity, implies $\lambda_0^2(S) = S$ for all $S \in \F_2[z]$ with $\deg S \leq 2$.
We will only use two facts about the series expansion $\sum_{n \geq 0} a(n) z^n \colonequal \frac{1}{R} = 1 + 1 z + 0 z^2 + 1 z^3 + 1 z^4 + 0 z^5 + \cdots$: it is periodic with period length $3$, and $a(2) = 0$.
We start by rewriting
\[
	\frac{z^j}{R}
	= \sum_{n \geq 0} a(n) z^{n + j}
	= \sum_{n \geq j} a(n - j) z^n.
\]
Since $\Lambda_0^2(z^n) = 0$ if $n \nequiv 0 \mod 4$, this implies
\[
	\Lambda_0^2\!\paren{\frac{z^j}{R}}
	= \Lambda_0^2\!\paren{\sum_{n \geq \ceil{j/4}} a(4 n - j) z^{4 n}}
	= \sum_{n \geq \ceil{j/4}} a(4 n - j) z^n.
\]
If $j = 0$ or $j = 1$, then $\ceil{j/4} = j$, so this series is $\sum_{n \geq j} a(4 n - j) z^n$.
If $j = 2$, then the coefficient for $n = \ceil{j/4} = 1$ is $a(4 \cdot 1 - j) = a(2) = 0$, so again the series is $\sum_{n \geq 2} a(4 n - j) z^n = \sum_{n \geq j} a(4 n - j) z^n$.
Since $a(n)_{n \geq 0}$ is periodic with period length $3$, we have $a(4 n - j) = a((4 n - j) \bmod 3) = a(n - j)$ for all $n \geq j \geq 0$.
Therefore
\[
	\Lambda_0^2\!\paren{\frac{z^j}{R}}
	= \sum_{n \geq j} a(4 n - j) z^n
	= \sum_{n \geq j} a(n - j) z^n
	= \sum_{n \geq 0} a(n) z^{n + j}
	= \frac{z^j}{R},
\]
as desired.
\end{example}

In general, periodicity of the series expansion of $\frac{1}{R}$ is guaranteed by the following standard argument.

\begin{lemma}\label{periodic series}
Let $R \in \field[z]$ be a polynomial with $\deg R \geq 1$.
If the coefficient of $z^0$ is nonzero, then $\frac{1}{R}$ has a power series expansion, and the sequence of coefficients of $\frac{1}{R}$ is periodic.
\end{lemma}

\begin{proof}
The fact that $\frac{1}{R}$ has a power series expansion follows from $R(0) \neq 0$ and the geometric series formula.
Write $\frac{1}{R} = \sum_{n \geq 0} a(n) z^n$.
The relation $R \sum_{n \geq 0} a(n) z^n = 1$ gives a recurrence for the coefficient sequence $a(n)_{n \geq 0}$.
Since there are only finitely many $(\deg R)$-tuples of elements from $\field$, the sequence $a(n)_{n \geq 0}$ is eventually periodic.
Since the coefficient of $z^{\deg R}$ is invertible, we can run the recurrence backward as well as forward, so $a(n)_{n \geq 0}$ is periodic.
\end{proof}

The main result of this section is that the size of the orbit under $\lambda_0$ is related to the factorization of $R$.
The \emph{factorization into irreducibles} of an element $R \in z^{-1} \field[z]$ is $R = c z^{e_0} R_1^{e_1} \cdots R_k^{e_k}$, where $z, R_1, \dots, R_k \in \field[z]$ are distinct, monic, irreducible polynomials, $c \in \field$, $e_0 \geq -1$, and $e_i \geq 1$ for all $i \in \{1, \dots, k\}$.
We say that $R$ is \emph{square-free} if $e_i = 1$ for all $i \in \{1, \dots, k\}$.
If $R \in z^{-1} \field[z]$ and $R \neq 0$, define $\deg R$ to be the largest exponent of $z$ with a nonzero coefficient in the expansion of $R$ in the monomial basis.

First we establish a bound on the orbit size for certain square-free Laurent polynomials $R$ with positive degree.
We use the convention that $\lcm() = 1$.

\begin{theorem}\label{square-free orbit size upper bound - positive degree}
Let $R \in z^{-1} \field[z]$ be a nonzero square-free Laurent polynomial such that $\deg R \geq 1$, whose factorization into irreducibles is of the form $c z^{e_0} R_1 \cdots R_k$, where $e_0 \in \{-1, 0\}$.
Let $\ell = \lcm(\deg R_1, \dots, \deg R_k)$.
Define $\lambda_0$ on $\field[z]$ by $\lambda_0(S) = \Lambda_0\!\paren{S R^{q - 1}}$.
Then $\lambda_0^\ell(S) = S$ for all $S \in \field[z]$ with $\deg S \leq \deg R$.
\end{theorem}

To prove Theorem~\ref{square-free orbit size upper bound - positive degree}, we use the following classical result to bound the period length of the series expansion of $\frac{1}{R}$ and to conclude that certain coefficients are $0$.

\begin{proposition}\label{product of monic irreducibles}
Let $\ell \geq 1$.
The product of all monic irreducible polynomials in $\field[z]$ with degree dividing $\ell$ is $z^{q^\ell} - z$.
\end{proposition}

Proposition~\ref{product of monic irreducibles} follows from the fact that $\F_{q^\ell}$ is the splitting field of $z^{q^\ell} - z$ over $\field$;
since each element in $\F_{q^\ell}$ has a minimal polynomial over $\field$, the product of all those minimal polynomials is $z^{q^\ell} - z$.

Now we prove Theorem~\ref{square-free orbit size upper bound - positive degree}.

\begin{proof}[Proof of Theorem~\ref{square-free orbit size upper bound - positive degree}]
Let $r \colonequal \deg R$.
Since $e_0 \in \{-1,0\}$, we have a power series expansion $\frac{1}{R} = \sum_{n \geq 0} a(n) z^n \in \field\doublebracket{z}$.
Let $j \in \{0, 1, \dots, r\}$.
By Proposition~\ref{Cartier}, $\lambda_0(z^j) = \Lambda_0(z^j R^{q - 1}) = \Lambda_0(\frac{z^j}{R}) R$.
Therefore, by iterating, $\lambda_0^\ell(z^j) = \Lambda_0^\ell(\frac{z^j}{R}) R$.
We show $\Lambda_0^\ell(\frac{z^j}{R}) = \frac{z^j}{R}$; this implies $\lambda_0^\ell(z^j) = z^j$, and the statement will follow from the linearity of $\lambda_0$.
Since $\Lambda_0^\ell(z^n) = 0$ if $n \nequiv 0 \mod q^\ell$, we have
\begin{align*}
	\Lambda_0^\ell\!\paren{\frac{z^j}{R}}
	&= \Lambda_0^\ell\!\paren{\sum_{n \geq j} a(n - j) z^n}
	= \Lambda_0^\ell\!\paren{\sum_{n \geq \ceil{j/q^\ell}} a(q^\ell n - j) z^{q^\ell n}} \\
	&= \sum_{n \geq \ceil{j/q^\ell}} a(q^\ell n - j) z^n.
\end{align*}
We will use the fact that the series expansion of $\frac{1}{R}$ is periodic to rewrite $a(q^\ell n - j)$.
Since each $\deg R_k$ divides $\ell$, Proposition~\ref{product of monic irreducibles} implies that the polynomial $z^{-e_0} R$ divides $z^{q^\ell - 1} - 1$.
Write $1 - z^{q^\ell - 1} = R T$ where $T \in z^{-e_0} \field[z]$; then
\begin{equation}\label{rational expression}
	\frac{1}{R} = \frac{T}{1 - z^{q^\ell - 1}}.
\end{equation}
Since $r \geq 1$, $a(n)_{n \geq 0}$ is periodic by Lemma~\ref{periodic series}.
Moreover, $\deg T < q^\ell - 1$, so its period length divides $q^\ell - 1$.
Therefore $a(q^\ell n - j) = a((q^\ell n - j) \bmod (q^\ell - 1)) = a(n - j)$ for all $n \geq j$, so
\[
	\sum_{n \geq j} a(q^\ell n - j) z^n
	= \sum_{n \geq j} a(n - j) z^n
	= \frac{z^j}{R},
\]
and it follows that
\[
	\Lambda_0^\ell\!\paren{\frac{z^j}{R}}
	= \sum_{n = \ceil{j/q^\ell}}^{j - 1} a(q^\ell n - j) z^n + \frac{z^j}{R}.
\]
It remains to show that $a(q^\ell n - j) = 0$ for all $n \in \{\ceil{j/q^\ell}, \dots, j - 2, j - 1\}$.
If $j = 0$ or $j = 1$, this is vacuously true, so assume $j \in \{2, 3, \dots, r\}$.
We identify certain $0$ coefficients in the series $\frac{1}{R}$.
From Equation~\eqref{rational expression}, we obtain
\[
	\sum_{n \geq 0} a(n) z^n
	= \frac{1}{R}
	= \frac{T}{1 - z^{q^\ell - 1}}
	= T + T z^{q^\ell - 1} + T z^{2 (q^\ell - 1)} + \cdots.
\]
Since $\deg T = q^\ell - 1 - \deg R = q^\ell - 1 - r$, this implies $0 = a(q^\ell - r) = a(q^\ell - r + 1) = \dots = a(q^\ell - 2)$; that is, $a(q^\ell - i) = 0$ for all $i \in \{2, 3, \dots, r\}$.
For all $n \in \{\ceil{j/q^\ell}, \dots, j - 2, j - 1\}$, we have $j - n + 1 \in \{2, 3, \dots, j - \ceil{j/q^\ell} + 1\} \subseteq \{2, 3, \dots, r\}$.
Therefore, since the period length of $a(n)_{n \geq 0}$ divides $q^\ell - 1$, we have $a(q^\ell n - j) = a(q^\ell - (j - n + 1)) = 0$ for $n \in \{\ceil{j/q^\ell}, \dots, j - 2, j - 1\}$, as desired.
\end{proof}

In Theorem~\ref{square-free orbit size upper bound - positive degree} we assumed that $\deg R \geq 1$.
However in general $\deg R \geq -1$; the next result extends Theorem~\ref{square-free orbit size upper bound - positive degree}.

\begin{corollary}\label{square-free orbit size upper bound}
Let $R \in z^{-1} \field[z]$ be a nonzero square-free Laurent polynomial whose factorization into irreducibles is of the form $c z^{e_0} R_1 \cdots R_k$, where $e_0 \in \{-1, 0\}$.
Let $\ell = \lcm(\deg R_1, \dots, \deg R_k)$.
Define $\lambda_0$ on $\field[z]$ by $\lambda_0(S) = \Lambda_0\!\paren{S R^{q - 1}}$.
Then $\lambda_0^\ell(S) = S$ for all $S \in \field[z]$ with $\deg S \leq \deg R$.
\end{corollary}

\begin{proof}
Let $r \colonequal \deg R$.
Theorem~\ref{square-free orbit size upper bound - positive degree} covers the case $r \geq 1$.
If $r = -1$, then $S = 0$ and the conclusion holds.

Suppose $r = 0$, so that $R = b z^{-1} + c$ for some $b, c \in \field$ with $c \neq 0$.
Here $\ell = 1$.
Let $S \in \field$.
By Proposition~\ref{Cartier}, $\lambda_0(S) = \Lambda_0(S R^{q - 1}) = S \Lambda_0(\frac{1}{R}) R$.
We show that $\Lambda_0(\frac{1}{R}) = \frac{1}{R}$, which will imply $\lambda_0(S) = S$.
If $b = 0$, then
\[
	\Lambda_0\!\left( \frac{1}{R} \right)
	= \Lambda_0\!\left( \frac{1}{c} \right)
	= \frac{1}{c}
	= \frac{1}{R}.
\]
If $b\neq 0$, then
\begin{align*}
	\Lambda_0\!\left( \frac{1}{R} \right)
	&= \Lambda_0\!\left( \frac{z}{b (1 - (-c/b) z)} \right)
	= \Lambda_0\!\left(\frac{1}{b} \sum_{n\geq 0} (-c/b)^n z^{n+1} \right) \\
	&= \frac{1}{b} \sum_{n\geq 1} (-c/b)^{nq - 1} z^n
	= \frac{1}{b} \sum_{n\geq 0} (-c/b)^n z^{n + 1}
	= \frac{1}{R}
\end{align*}
since $n q - 1 \equiv n - 1 \mod (q - 1)$.
\end{proof}

Proposition~\ref{univariate emulation} tells us that the orbit of $y \frac{\partial P}{\partial y}$ under $\lambda_{0,0}$, when restricted to the left, right, and top borders of $V$, is dictated by the orbits of its projection onto these borders under $\lambda_0$, where the latter is defined using $A_0/y$, $A_h/y$, and $B_d$.
These orbits can be studied using Corollary~\ref{square-free orbit size upper bound} as follows.

\begin{example}\label{example orbits}
We continue to use the polynomial $P$ from Examples~\ref{example} and \ref{example projections}.
The initial state is $S_0 = y \frac{\partial P}{\partial y}$, and the second state in the orbit of $S_0$ under $\lambda_{0, 0}$ is
\[
	S_1
	\colonequal \lambda_{0, 0}(S_0)
	= x y^3 + (x^2 + x + 1) y^2 + (2 x^2 + 2) y + x^2 + x.
\]
In Example~\ref{example denominator projections}, we reduced to the univariate operator $\lambda_0$ with $R = A_0/y$, $R = A_h/y$, and $R = B_d$.
We verify that the conditions of Corollary~\ref{square-free orbit size upper bound} hold.
The factorizations into irreducibles of the three Laurent polynomials $R$ are
\begin{align*}
	A_0/y &= 2 y^3 + y + 1 = 2 (y^3 + 2 y + 2) \\
	A_h/y &= y^3 + 1 + 2y^{-1} = y^{-1} (y^4 + y + 2) \\
	B_d &= x^2 + x + 2.
\end{align*}
Moreover, $\deg \pi\subl (S_1) \leq \deg(A_0/y)$, $\deg \pi\subr (S_1) \leq \deg(A_h/y)$, and $\deg \pi\subt (S_1) \leq \deg B_d$.
Therefore, by Corollary~\ref{square-free orbit size upper bound}, the three relevant orbits under the three operators $\lambda_0$ are periodic and have respective period lengths dividing $3$, $4$, and $2$.
For $R = A_0/y$, the orbit of $\pi\subl (S_1)$ under $\lambda_0$ is
\[
	y^2 + 2 y, \, y^2 + y, \, y^2, \, y^2 + 2 y, \, \dots.
\]
For $R = A_h/y$, the orbit of $\pi\subr (S_1)$ under $\lambda_0$ is
\[
	y^2 + 2 y + 1, \, y^2 + y + 1, \,y^2, \, 1, \, y^2 + 2 y + 1, \, \dots.
\]
Lastly, for $R = B_d$, the orbit of $\pi\subt (S_1)$ under $\lambda_0$ is
\[
	x, \, 2 x, \, x, \, \dots.
\]
In particular, the upper bounds on the period lengths are attained.
\end{example}

It remains to remove the restriction that $R$ is square-free.
Unlike the square-free case, the orbit of $S$ under $\lambda_0$ may have a transient (in other words, may be eventually periodic but not periodic).
First we give two propositions showing that elements sufficiently far out in the orbit are necessarily divisible by a certain polynomial; if $S$ is not divisible by this polynomial then the orbit has a transient.

\begin{proposition}\label{non-square-free transient size upper bound}
Let $R \in z^{-1} \field[z]$ be a nonzero Laurent polynomial such that $R = F^e G$ for some $F \in \field[z]$, $G \in z^{-1} \field[z]$, and $e\ge 1$.
For all $S \in \field[z]$ and all $n \geq \ceil{\log_q e}$, the polynomial $\lambda_0^n(S)$ is divisible by $F^{e - 1}$.
\end{proposition}

Note that there are potentially multiple ways to decompose $R$ in Proposition~\ref{non-square-free transient size upper bound}.
For example, if $R = z^4$ then we could write $F = z, e = 4, G = 1$ or $F = z, e = 5, G = z^{-1}$.
The latter choice leads to a stronger conclusion regarding divisibility.

\begin{proof}[Proof of Proposition~\ref{non-square-free transient size upper bound}]
Let $S \in \field[z]$, and write $S = F^s T$ for some $s \geq 0$.
(We do not require $s$ to be maximal.)
We have
\begin{align*}
	\lambda_0(S)
	= \Lambda_0(S R^{q - 1})
	&= \Lambda_0\!\paren{F^s T F^{e (q - 1)} G^{q - 1}} \\
	&= \Lambda_0\!\paren{F^{(s + e (q - 1)) \bmod q} T G^{q - 1}} F^{\floor{(s + e (q - 1))/q}}
\end{align*}
by Proposition~\ref{Cartier}.
Therefore $\lambda_0(S)$ is divisible by $F^{\floor{(s + e (q - 1))/q}} = F^{e + \floor{(s - e)/q}}$, so we iterate the function $f(x) = e + \floor{\frac{x - e}{q}}$.
Let $n \geq \ceil{\log_q e}$, so that $\floor{-\frac{e}{q^n}} = -1$.
As in the proof of Lemma~\ref{iterating floor}, we have $f^n(s) = e + \floor{\frac{s - e}{q^n}} \geq e + \floor{-\frac{e}{q^n}} = e - 1$.
Therefore $\lambda_0^n(S)$ is divisible by $F^{e - 1}$.
\end{proof}

\begin{example}
Let $q = 3$ and $R = z^{-1} (z + 1)^3 (z + 2)$.
The orbit of $S = 1$ under $\lambda_0$ is $1, (z + 1)^2, (z + 1)^2, \dots$.
It has transient length $1$ (and period length $1$).
\end{example}

If $F = z$ and if $G$ is a polynomial, then we can slightly increase the exponent to which $F$ eventually divides elements in the orbit under $\lambda_0$.
We do this by tracking $0$ coefficients rather than using Proposition~\ref{Cartier}, which we used to prove Proposition~\ref{non-square-free transient size upper bound}.

\begin{proposition}\label{power of y transient size upper bound}
Let $R \in \field[z]$ be a nonzero polynomial such that $R = z^e G$ for some $G \in \field[z]$ where $e \geq 1$ and $G$ is not divisible by $z$.
For all $S \in \field[z]$ and all $n \geq \floor{\log_q e} + 1$, the polynomial $\lambda_0^n(S)$ is divisible by $z^e$.
\end{proposition}

\begin{proof}
We use the fact that if a polynomial is divisible by $z^\nu$ then, when we apply $\Lambda_0$, we get a polynomial that is divisible by $z^{\ceil{\nu/q}}$.
Let $S \in \field[z]$, and write $S = z^s T$ where $s \geq 0$ and $T$ is not divisible by $z$.
We have
\[
	\lambda_0(S)
	= \Lambda_0(S R^{q - 1})
	= \Lambda_0\!\paren{z^{s + e (q - 1)} T G^{q - 1}}.
\]
Since $z^{s + e (q - 1)} T G^{q - 1}$ is divisible by $z^{e q + s - e}$, it follows that $\lambda_0(S)$ is divisible by $z^{e + \ceil{\frac{s - e}{q}}}$.
Let $f(x) = e + \ceil{\frac{x - e}{q}}$.
Applying Lemma~\ref{iterating floor}, if $n \geq \floor{\log_q e} + 1$ then $\lambda_0^n(S)$ is divisible by $z^e$.
\end{proof}

In the following theorem, we show that the situation for a general (not necessarily square-free) element $R \in z^{-1} \field[z]$ reduces to Corollary~\ref{square-free orbit size upper bound} by Propositions~\ref{non-square-free transient size upper bound} and \ref{power of y transient size upper bound}.
The idea of the proof is that if $R$ is divisible by $F^e$, then every application of $\lambda_0$ pushes the image into a smaller vector space until we are emulating the map $\lambda_0$ for a square-free polynomial $R'$.
We define $\log_q 0 = - \infty$, $\floor{- \infty} = - \infty$, $\ceil{- \infty} = - \infty$, and $\max() = 0$.
(When $\deg R = -1$, the only polynomial $S$ satisfying $\deg S \leq \deg R$ is $S = 0$, so the theorem does not say much in this case.)

\begin{theorem}\label{univariate orbit size upper bound}
Let $R \in z^{-1} \field[z]$ be a nonzero Laurent polynomial.
Let $R = c z^{e_0} R_1^{e_1} \cdots R_k^{e_k}$ be its factorization into irreducibles.
Let
\begin{equation}\label{transient time}
	t = \max\!\left(
		\floor{\log_q \max(e_0, 0)} + 1,
		\ceil{\log_q \max(e_1, \dots, e_k)},
		0
	\right)
\end{equation}
and $\ell = \lcm(\deg R_1, \dots, \deg R_k)$.
Define $\lambda_0$ on $\field[z]$ by $\lambda_0(S) = \Lambda_0\!\paren{S R^{q - 1}}$.
For all $S \in \field[z]$ with $\deg S \leq \deg R$, the orbit size of $S$ under $\lambda_0$ is at most $t + \ell$.
\end{theorem}

\begin{proof}
Define the \emph{radical} of $R$ by $\rad R = c z^{\min(e_0, 0)} R_1 \cdots R_k$.
Let
\[
	U = z^{\max(e_0, 0)} R_1^{e_1 - 1} \cdots R_k^{e_k - 1}
\]
so that $U \rad R = R$.
Let $S \in \field[z]$ with $\deg S \leq \deg R$.
We show that the orbit of $S$ under $\lambda_0$ is eventually periodic with transient length at most $t$ and period length dividing $\ell$.

We claim that $\lambda_0^t(S) = T U$ for some $T \in \field[z]$ satisfying $\deg T \leq \deg \rad R$.
To see that $R_i^{e_i - 1}$ divides $\lambda_0^t(S)$ for each $i \in \{1, 2, \dots, k\}$, we apply Proposition~\ref{non-square-free transient size upper bound} with $F = R_i$.
If $e_0 \in \{-1, 0\}$, then $z^{\max(e_0, 0)} = 1$.
If $e_0 \geq 1$, then Proposition~\ref{power of y transient size upper bound} implies that $z^{\max(e_0, 0)} = z^{e_0}$ divides $\lambda_0^t(S)$.
To see that $\deg T \leq \deg \rad R$, we first observe that the definition of $\lambda_0$ and the assumption that $\deg S \leq \deg R$ imply $\deg \lambda_0^t(S) \leq \deg R$.
Therefore
\begin{align*}
	\deg T
	&= \deg \lambda_0^t(S) - \deg U \\
		& \leq \deg R - \deg U \\
	&= \paren{e_0 + \sum_{i = 1}^k e_i \deg R_i} -\paren{\max(e_0, 0) + \sum_{i = 1}^k (e_i - 1) \deg R_i} \\
	&= \min(e_0, 0) + \sum_{i = 1}^k \deg R_i
	= \deg \rad R.
\end{align*}
This completes the proof of the claim.

Next we use the identity $e_i - 1 + e_i (q - 1) = e_i q - 1 = q - 1 + (e_i - 1) q$.
For all $T \in \field[z]$ (and in particular for the $T$ satisfying $\lambda_0^t(S) = T U$),
\begin{align*}
	\lambda_0(T U)
	= \Lambda_0\!\paren{T U R^{q - 1}}
	&= \Lambda_0\!\paren{T c^{q - 1} z^{\max(e_0, 0) + e_0 (q - 1)} R_1^{e_1 q - 1} \cdots R_k^{e_k q - 1}} \\
	&= \Lambda_0\!\paren{T c^{q - 1} z^{\min(e_0, 0) (q - 1)} R_1^{q - 1} \cdots R_k^{q - 1} U^q} \\
	&= \Lambda_0\!\paren{T \, (\rad R)^{q - 1}} U
\end{align*}
by Proposition~\ref{Cartier}.
Accordingly, define $\kappa_0 \colon \field[z] \to \field[z]$ by $\kappa_0(T) = \Lambda_0(T \, (\rad R)^{q - 1})$, so that $\lambda_0(T U) = \kappa_0(T) U$.
Iterating, we have $\lambda_0^\ell(T U) = \kappa_0^\ell(T) U$.
Applying Corollary~\ref{square-free orbit size upper bound} to $\kappa_0$, we have $\kappa_0^\ell(T) = T$ since $\rad R$ is square-free and $\deg T \leq \deg \rad R$.
Therefore
\[
	\lambda_0^{t + \ell}(S)
	= \lambda_0^\ell\!\paren{\lambda_0^t(S)}
	= \lambda_0^\ell\!\paren{T U}
	= \kappa_0^\ell(T) U
	= T U
	= \lambda_0^t(S),
\]
so the orbit of $S$ under $\lambda_0$ contains at most $t + \ell$ elements.
\end{proof}

\section{Orbit size under $\lambda_{0, 0}$}\label{section: orbit size}

In this section, we prove Theorem~\ref{kernel size upper bound}.
Our aim is to bound the size of the $q$-kernel for $F = \sum_{n \geq 0} a(n) x^n \in \field\doublebracket{x}$, which satisfies $P(x, F) = 0$.
From Corollary~\ref{kernel preliminary upper bound}, it remains to bound $\size{\orb_{\Lambda_0}(F)}$, equivalently, $\size{\orb_{\lambda_{0,0}}(S_0)}$ where $S_0 = y \frac{\partial P}{\partial y}$.

To do this, we will use Theorem~\ref{univariate orbit size upper bound} to obtain bounds on orbit sizes under $\lambda_{0,0}$.
We need the following lemma, which bounds the degree of the three border polynomials of $\lambda_{0,0}\!\left(y \frac{\partial P}{\partial y}\right)$.
Recall the definitions of $A_i$ and $B_j$ from Equation~\eqref{equation: P form}, that $\deg(A_0/y) \geq 0$, and that $A_h$ and $B_d$ are nonzero.

\begin{lemma}\label{right degree}
Let $S_0 = y \frac{\partial P}{\partial y}$.
Then
\begin{enumerate}
\item
$\deg \pi\subl(\lambda_{0, 0}(S_0)) \leq \deg(A_0/y)$,
\item
$\deg \pi\subr(\lambda_{0, 0}(S_0)) \leq \deg(A_h/y)$, and
\item
$\deg \pi\subt (\lambda_{0,0}^n(S_0)) \leq \deg B_d$ for all $n \geq \floor{\log_q h} + 2$.
\end{enumerate}
\end{lemma}

\begin{proof}
For the first two statements, we will use
\[
	\pi\subl(S_0) = y \tfrac{d A_0}{d y} \mbox{\quad and \quad}
	\pi\subr(S_0) = y \tfrac{d A_h}{d y}.
\]

For the first statement, let $R = A_0/y$.
Part~\ref{left border} of Proposition~\ref{univariate emulation} gives $\pi\subl(\lambda_{0, 0}(S_0)) = \lambda_0(\pi\subl(S_0)) = \Lambda_0\!\paren{y \tfrac{d A_0}{d y} \cdot (A_0/y)^{q - 1}}$.
The degree of this polynomial is at most $\deg(A_0/y)$.

The second statement follows in the same way by applying Part~\ref{right border} of Proposition~\ref{univariate emulation} since $\deg_x S_0 \leq h$.

For the third statement, let $R = B_d$.
We cannot apply Part~\ref{top border} of Proposition~\ref{univariate emulation} to $S_0$, since we cannot guarantee $\deg_y S_0 \leq d - 1$.
Instead, we apply it to $\lambda_{0, 0}(S_0)$ since $\deg_y \lambda_{0, 0}(S_0) \leq d - 1$ by Proposition~\ref{transient}.
Let $
	S
	= \pi\subt(\lambda_{0, 0}(S_0))
	= \pi\subt\!\paren{\Lambda_{0, 0}\!\paren{S_0 Q^{q - 1}}}$.
All nonzero monomials in $\pi\subt\!\paren{\Lambda_{0, 0}\!\paren{S_0 Q^{q - 1}}}$ come from applying $\Lambda_0$ to terms in $(d - 1) B_{d - 1} \cdot B_d^{q - 1}$ or $d B_d \cdot (q - 1) B_{d - 1} B_d^{q - 2}$, so $S = \Lambda_0\!\paren{(q d - 1) B_{d - 1} B_d^{q - 1}}$.
We have $r \colonequal \deg R = \deg B_d \geq 0$ and $s \colonequal \deg S \leq h$.
By Proposition~\ref{fixed point}, if $n - 1 \geq \floor{\log_q h} + 1 \geq \floor{\log_q\max(s - r,1)} + 1$ then $\deg \pi\subt(\lambda_{0, 0}^n(S_0)) \leq \deg B_d$.
\end{proof}

\begin{example}\label{example transient}
For the polynomial $P$ in Examples~\ref{example} and \ref{example orbits}, it suffices to take $n = 1$ to achieve $\deg \pi\subt (\lambda_{0,0}^n(S_0)) \leq \deg B_d$, since
\[
	\deg \pi\subt(\lambda_{0,0}(S_0))
	= \deg \pi\subt(S_1)
	= \deg x
	\leq \deg(x^2 + x + 2)
	= \deg B_d.
\]
\end{example}

The eventual period lengths in Theorem~\ref{univariate orbit size upper bound} are bounded by $\lcm(\deg R_1, \dots, \deg R_k)$.
We will use the function $\Landaulcm(l, m, n)$ defined in Section~\ref{section: introduction} to obtain a bound that is independent of the factorizations of $A_0/y$, $A_h/y$, and $B_d$.
We rephrase Theorem~\ref{kernel size upper bound} in terms of $\ker_q(a(n)_{n\geq 0})$, since it has the same size as the minimal automaton for $a(n)_{n\geq 0}$.

\begin{main theorem}
Let $F = \sum_{n \geq 0} a(n) x^n \in \field\doublebracket{x} \setminus \{0\}$ be the Furstenberg series associated with a polynomial $P \in \field[x, y]$ of height $h$ and degree $d$.
Then
\[
	\size{\ker_q(a(n)_{n\geq 0})}
	\leq q^{h d} + q^{(h - 1) (d - 1)} \Landaulcm(h, d, d) + \floor{\log_q h} + \floor{\log_q \max(h, d)} + 3.
\]
\end{main theorem}

\begin{proof}
By Corollary~\ref{kernel preliminary upper bound}, $\size{\ker_q(a(n)_{n\geq 0})} \leq q^{h d} + \size{\orb_{\Lambda_0}(F)}$, so we now bound $\size{\orb_{\Lambda_0}(F)} \leq \size{\orb_{\lambda_{0, 0}}(S_0)}$.
We do this by emulating $\lambda_{0, 0}$ with the appropriate univariate operators $\lambda_0$ on the left, right, and top borders of $V$ and using a crude upper bound for the rest.
Lemma~\ref{right degree} and Proposition~\ref{univariate emulation} will allow us to do this.

We use the following fact.
Let $V$ be a finite vector space with basis $\mathcal B$.
Let $(\mathcal B_1, \mathcal B_2)$ be a partition of $\mathcal B$, and let $U_1$ and $U_2$ be the subspaces generated by $\mathcal B_1$ and $\mathcal B_2$.
Let $\pi_{U_i}$ denote projection onto $U_i$ (for $i\in\{1,2\}$).
If $f \colon V \to V$ and $\tilde f \colon U_1 \to U_1$ are linear transformations satisfying $\pi_{U_1} \circ f = \tilde f \circ \pi_{U_1}$, then
\[
	f(x) = \pi_{U_1}(f(x)) + \pi_{U_2}(f(x)) = \tilde{f}(\pi_{U_1}(x)) + \pi_{U_2}(f(x)),
\]
so that $\size{\orb_f(x)} \leq \size{U_2} \cdot \size{\orb_{\tilde f} (\pi_{U_1}(x))}$ for all $x \in V$.

We apply this fact to $U_2 = V^\circ$, where $V^\circ$ is defined in Equation~\eqref{V interior definition}, and $f=\lambda_{0,0}$.
Here $U_1$ is the space generated by the union of the generators of the left, right, and top borders of $V$:
\[
	U_1= \left\langle \{x^0 y^j : 0 \leq j \leq d - 1 \} \cup \{x^h y^j : 0 \leq j \leq d - 1 \} \cup \{x^i y^{d-1} : 1 \leq i \leq h - 1 \}\right\rangle.
\]
To define $\tilde f$, note that Proposition~\ref{univariate emulation} gives us an operator $\tilde \lambda_0 \colon U_1 \to U_1$ that satisfies $\pi_{U_1} \circ \lambda_{0, 0} = \tilde \lambda_0 \circ \pi_{U_1}$.
(The operator $\tilde \lambda_0$ acts as the appropriate $\lambda_0$ on the three respective borders.)
Set $\tilde f = \tilde \lambda_0$.
The fact in the previous paragraph now implies $\size{\orb_{\lambda_{0,0}}(x)} \leq \size{V^\circ} \cdot \size{\orb_{\tilde \lambda_0} (\pi_{U_1}(x))}$ for all $x \in V$.
Let
\[
	t \colonequal \floor{\log_q h} + 2 + \floor{\log_q \max(h, d)} + 1;
\]
we will justify the definition of $t$ below.
Set $S_t\colonequal \lambda_{0,0}^t(y \frac{\partial P}{\partial y})$.
We have $S_t \in V$ by Proposition~\ref{transient} since $t\geq 1$.
Write $S_t = \pi_{V^\circ}(S_t) + T$ where $T \in U_1$.
Therefore, using $\size{V^\circ} = q^{(h - 1) (d - 1)}$, we have
\begin{equation}\label{intermediate}
	\size{\orb_{\lambda_{0,0}}(S_t)}
	\leq q^{(h - 1) (d - 1)} \cdot \size{\orb_{\tilde \lambda_0}(T)}.
\end{equation}
It remains to bound $\size{\orb_{\tilde \lambda_0}(T)}$.
We will do this by bounding the orbit sizes of the projections $\pi\subl(T)$, $\pi\subr (T)$, and $\pi\subt (T)$ under the respective operators $\lambda_0$, defined by the Laurent polynomials $R = A_0/y$ on the left border, $R = A_h/y$ on the right border, and $R = B_d$ on the top border.
Recall that $\deg(A_0/y) \leq d - 1$, $\deg(A_h/y) \leq d - 1$, and $\deg B_d \leq h$.
Equation~\eqref{transient time} in Theorem~\ref{univariate orbit size upper bound} will give transient lengths $t\subl$, $t\subr$, and $t\subt$ in terms of the factorization of $R = c z^{e_0} R_1^{e_1} \cdots R_k^{e_k}$.
These transient lengths are at most
\begin{align}\label{max transient}
	\max(t\subl, t\subr, t\subt)
	& \leq \max\!\left(
		\floor{\log_q \max(h,d-1)} + 1,
		\ceil{\log_q \max(h,d)}
	\right) \nonumber \\
	& \leq \floor{\log_q \max(h,d)} + 1;
\end{align}
here, the term $\floor{\log_q \max(h,d-1)} + 1$ involves $d-1$ whereas $\ceil{\log_q \max(h,d)}$ involves $d$ because $R \in z^{-1} \field[z]$ and $e_0$ is the exponent of the $z$ factor.
Since $t \geq \floor{\log_q h} + 2$, Lemma~\ref{right degree} and Proposition~\ref{univariate emulation} tell us that, on $U_1$, we can emulate the action of $\lambda_{0, 0}$ on $S_t$ with the three operators $\lambda_0$.
Since $\pi\subl(T) = \pi\subl(S_t)$, $\pi\subr(T) = \pi\subr(S_t)$, and $\pi\subt(T) = \pi\subt(S_t)$, we consider the sizes of the orbits
\begin{align*}
	\orb\subl(S_t) &= \{\lambda_0^n(\pi\subl(S_t)) : n\geq 0\} \\
	\orb\subr(S_t) &= \{\lambda_0^n(\pi\subr(S_t)) : n\geq 0\} \\
	\orb\subt(S_t) &= \{\lambda_0^n(\pi\subt(S_t)) : n\geq 0\}.
\end{align*}
We claim that these three orbits are periodic, i.e.\ have no transient;
this follows from the proof of Theorem~\ref{univariate orbit size upper bound} since $t \geq \floor{\log_q \max(h, d)} + 1$ and this is the maximum transient length by~\eqref{max transient}.
Lemma~\ref{right degree} implies $\deg S_t \leq \deg B_d$, so we can apply Theorem~\ref{univariate orbit size upper bound} with $R = B_d$ to $\pi\subt(S_t)$.
It tells us that $\size{\orb\subt(S_t)} = \lcm(\sigma)$ for some integer partition $\sigma \in \partitions(\deg B_d)$.
Similarly, for $\size{\orb\subl(S_t)}$ and $\size{\orb\subr(S_t)}$ we obtain integer partitions in $\partitions(1+\deg A_0/y ) = \partitions(\deg A_0)$ and $\partitions(1+\deg A_h/y) = \partitions(\deg A_h)$.

We now use
\begin{equation}\label{lcm upper bound}
	\size{\orb_{\tilde \lambda_0}(T)}
	\leq \lcm\!\big(\size{\orb\subl(S_t)}, \size{\orb\subr(S_t)}, \size{\orb\subt(S_t)}\big)
\end{equation}
and maximize over the orbit sizes that arise.
By Equation~\eqref{lcm upper bound} and the definition of $\Landaulcm$, we have $\size{\orb_{\tilde \lambda_0}(T)}\leq \Landaulcm(h, d, d)$ since $\deg A_0 \leq d$, $\deg A_h \leq d$, and $\deg B_d \leq h$.
Equation~\eqref{intermediate} gives
\[
	\size{\orb_{\lambda_{0,0}}(S_t)}
	\leq q^{(h - 1) (d - 1)} \Landaulcm(h, d, d).
\]
It follows that $\size{\orb_{\Lambda_0}(F)} \leq \size{\orb_{\lambda_{0,0}} (S_t)} + t \leq q^{(h - 1) (d - 1)} \Landaulcm(h, d, d) + t$ as desired.
\end{proof}

\begin{example}
We continue Examples~\ref{example} and \ref{example transient}, where $h = 2$ and $d = 4$.
We have $\Landaulcm(h, d, d) = 12 = \lcm(3, 4, 2)$.
Computing the orbit of $S_0$ under $\lambda_{0, 0}$, one finds that it has size $157$, consisting of $1$ transient state followed by a period with length $156 = 13 \cdot 12$.
This period length is less than the theoretical maximum $q^{(h - 1) (d - 1)} \Landaulcm(h, d, d) = 27 \cdot 12$.
The number of states in the constructed automaton is $5989 \approx 3^{7.917}$, which is on the order of the upper bound
\begin{align*}
	q^{h d} &+ q^{(h - 1) (d - 1)} \Landaulcm(h, d, d) + \floor{\log_q h} + \floor{\log_q \max(h, d)} + 3 \\
	&= 3^8 + 3^3 \cdot 12 + 4 \\
	&= 6889
	\approx 3^{8.044}.
\end{align*}
Minimizing the automaton reduces the number of states by $1$ to $5988$.
\end{example}

Asymptotically, we have the following.

\begin{main theorem - asymptotic}
Let $F = \sum_{n \geq 0} a(n) x^n \in \field\doublebracket{x}$ be the Furstenberg series associated with a polynomial $P \in \field[x, y]$ of height $h$ and degree $d$.
Then $\size{\ker_q(a(n)_{n\geq 0})}$ is in $(1 + o(1)) q^{h d}$ as any of $q$, $h$, or $d$ tends to infinity and the others remain constant.
\end{main theorem - asymptotic}

\begin{proof}
Recall that the conditions on a Furstenberg series guarantee that $d \geq 1$.
If $h = 0$, then the power series $F$ is the $0$ series, so $\size{\ker_q(a(n)_{n\geq 0})} = 1$.
Therefore we assume $h \geq 1$.

As before, let $\Landau(n)$ be the Landau function.
The set of triples of integer partitions of $h,d,d$ gives rise to a subset of integer partitions of $h+2d$.
Thus $\Landaulcm(h, d, d) \leq \Landau(h + 2 d)$.
By Theorem~\ref{kernel size upper bound},
\[
	\size{\ker_q(a(n)_{n\geq 0})}
	\leq q^{h d} + q^{(h - 1) (d - 1)} \Landau(h + 2 d) + \floor{\log_q h} + \floor{\log_q \max(h, d)} + 3.
\]
The expression $\floor{\log_q h} + \floor{\log_q \max(h, d)} + 3$ is clearly in $o(1) q^{h d}$.
It remains to show that $q^{(h - 1) (d - 1)} \Landau(h + 2 d)$ is also in $o(1) q^{h d}$.
Landau~\cite{Landau} proved that $\log \Landau(n) \sim \sqrt{n \log n}$, that is, $\Landau(n) = e^{(1 + \epsilon(n)) \sqrt{n \log n}}$, where $\epsilon(n) \to 0$ as $n \to \infty$.
Therefore
\[
	\frac{q^{(h - 1) (d - 1)} \Landau(h + 2 d)}{q^{h d}}
	= \frac{\Landau(h + 2 d)}{q^{h + d - 1}}
	= \frac{e^{(1+\epsilon(h + 2 d)) \sqrt{(h + 2 d) \log (h + 2 d)}}}{q^{h + d - 1}},
\]
and this tends to $0$ as any of $q$, $h$, or $d$ tends to infinity and the others remain constant.
\end{proof}

Bridy used a similar argument, also bounding the orbit size by $\Landau(h + 2 d)$~\cite[Proof of Theorem 1.2]{Bridy}.

\begin{example}
The factor $1 + o(1)$ cannot be removed from the bound in Theorem~\ref{kernel size asymptotic bound}.
Let $q = 2$, and consider
\[
	P = (x^3 + x^2 + 1) y^3 + (x^3 + 1) y^2 + (x^3 + x^2 + x + 1) y + x^3 + x^2 \in \F_2[x, y]
\]
with height~$h = 3$ and degree~$d = 3$.
The coefficient sequence $a(n)_{n \geq 0}$ of the series $F \in \F_2\doublebracket{x}$ satisfying $P(x, F) = 0$ is $0, 0, 1, 0, 0, 1, 0, 0, 0, 0, 1, 0, 1, 1, 0, 0, \dots$.
The constructed automaton has $532$ states.
Minimizing reduces the number of states by only $1$ to $531$, which is larger than $q^{h d} = 512$.
\end{example}

With the same techniques as in the proof of Theorem~\ref{kernel size upper bound}, one obtains the following result, which concerns diagonals of rational functions that are not necessarily of the form in Theorem~\ref{Furstenberg}.
To state it, we extend the function $\Landaulcm(n_1, n_2, n_3)$ from Section~\ref{section: introduction} to $\Landaulcm(n_1, n_2, n_3, n_4)$, defined analogously as the maximum value of $\lcm(\lcm(\sigma_1), \lcm(\sigma_2), \lcm(\sigma_3), \lcm(\sigma_4))$ over integer partitions $\sigma_i$ of integers in $\{1, 2, \dots, n_i\}$.
The reason for this is that Theorem~\ref{kernel size upper bound - diagonals} is symmetric in $x$ and $y$, unlike Theorem~\ref{Furstenberg}.
This symmetry leads to the appearance of $\Landaulcm(h, h, d, d)$ in Theorem~\ref{kernel size upper bound - diagonals} instead of $\Landaulcm(h, d, d)$ as in Theorem~\ref{kernel size upper bound}.

\begin{theorem}\label{kernel size upper bound - diagonals}
Let $P(x, y)$ and $Q(x, y)$ be polynomials in $\field[x,y]$ such that $Q(0,0) \neq 0$.
Let
\[
	F = \mathcal{D}\!\paren{\frac{P(x, y)}{Q(x, y)}},
\]
and write $F(x) = \sum_{n \geq 0} a(n) x^n$.
Let
\begin{align*}
	h &= \max(\deg_x P, \deg_x Q ) \\
	d &= \max(\deg_y P , \deg_y Q),
\end{align*}
and assume $h \geq 1$ and $d \geq 1$.
Then the size of $\ker_q(a(n)_{n \geq 0})$ is at most
\[
	q^{h d} + q^{(h-1)(d-1)} \Landaulcm(h, h, d, d)
	+ 2 \floor{\log_q \max(h, d)}
	+ 2.
\]
Consequently, $\size{\ker_q(a(n)_{n \geq 0})}$ is in $(1 + o(1)) q^{hd}$ as any of $q$, $h$, or $d$ tends to infinity and the others remain constant.
\end{theorem}

When comparing Theorem~\ref{kernel size upper bound - diagonals} to the rest of the paper, note that a Furstenberg series is the diagonal of a rational function, whose denominator is also called $Q$, but in Theorem~\ref{kernel size upper bound - diagonals} the denominator $Q$ has degree $d$ and not $d - 1$ as before.

The structure of the proof of Theorem~\ref{kernel size upper bound - diagonals} is similar to that of Theorem~\ref{kernel size upper bound}.
One difference is that the diagonal in Theorem~\ref{Furstenberg} contains expressions of the form $P(x y, y)$, which led us to shear and to consider the maps $\lambda_{r, 0}$ on Laurent polynomials.
For general diagonals, the symmetry in $x$ and $y$ means that no shearing is required, the relevant maps are $\lambda_{r, r}$, and no Laurent polynomials enter the picture.
We define the main objects and state the modifications of relevant results used in the proof.
With $h$ and $d$ defined as in Theorem~\ref{kernel size upper bound - diagonals}, let
\[
	V \colonequal \left\langle x^i y^j : \text{$0 \leq i \leq h$ and $0 \leq j \leq d $}\right\rangle
\]
and
\begin{equation*}
	V^\circ = \left\langle x^i y^j : \text{$1 \leq i \leq h - 1$ and $1 \leq j \leq d - 1$}\right\rangle.
\end{equation*}
The initial state of the automaton is $S_0 = P$.
Define $\pi\subl$ and $\pi\subr$ as in Section~\ref{section: structure}.
For a polynomial $S = \sum_{i, j} c_{i,j} x^i y^j$, define $\pi\subb(S) = \sum_i c_{i,0} x^i$ and $\pi\subt(S) = \sum_i c_{i,d} x^i$.
The following results are analogues of Proposition~\ref{univariate emulation} and Lemma~\ref{right degree}; their proofs are similar.

\begin{proposition}\label{univariate emulation - diagonals}
We have the following.
\begin{enumerate}
\item
Let $R =\pi\subl(Q)$.
For all $S \in \field[x, y]$,
\[
	\pi\subl(\lambda_{0, 0}(S))
	= \lambda_0(\pi\subl(S)).
\]
\item
Let $R =\pi\subr(Q)$.
For all $S \in \field[x, y]$ with height at most $h$,
\[
	\pi\subr(\lambda_{0, 0}(S))
	= \lambda_0(\pi\subr(S)).
\]
\item
Let $R = \pi\subb(Q)$.
For all $S \in \field[x, y]$,
\[
	\pi\subb(\lambda_{0, 0}(S))
	= \lambda_0(\pi\subb(S)).
\]
\item
Let $R = \pi\subt(Q)$.
For all $S \in \field[x, y]$ with degree at most $d$,
\[
	\pi\subt(\lambda_{0, 0}(S))
	= \lambda_0(\pi\subt(S)).
\]
\end{enumerate}
\end{proposition}

\begin{lemma}\label{right degree - diagonals}
Let
\begin{align*}
	u\subl &= \floor{\log_q \max\!\paren{d - \deg\pi\subl(Q), 1}} + 1 \\
	u\subr &= \floor{\log_q \max\!\paren{d - \deg\pi\subr(Q), 1}} + 1 \\
	u\subb &= \floor{\log_q \max\!\paren{h - \deg \pi\subb(Q), 1}} + 1\\
	u\subt &= \floor{\log_q \max\!\paren{h - \deg \pi\subt(Q), 1}} + 1.
\end{align*}
For all $S \in V$, we have
\begin{enumerate}
\item
$\deg\pi\subl(\lambda_{0, 0}^n(S)) \leq \deg\pi\subl(Q)$ for all $n \geq u\subl$,
\item
$\deg \pi\subr(\lambda_{0, 0}^n(S)) \leq \deg\pi\subr(Q)$ for all $n \geq u\subr$,
\item
$\deg \pi\subb(\lambda_{0,0}^n(S)) \leq \deg \pi\subb(Q)$ for all $n \geq u\subb$,
and
\item
$\deg \pi\subt(\lambda_{0,0}^n(S)) \leq \deg \pi\subt(Q)$ for all $n \geq u\subt$.
\end{enumerate}
\end{lemma}

With Proposition~\ref{univariate emulation - diagonals} and Lemma~\ref{right degree - diagonals}, one follows the proof of Theorem~\ref{kernel size upper bound} to prove Theorem~\ref{kernel size upper bound - diagonals}.
One of the terms $\floor{\log_q \max(h, d)} + 1$ in Theorem~\ref{kernel size upper bound - diagonals} comes from Lemma~\ref{right degree - diagonals} by bounding $u\subl, u\subr$ by $\floor{\log_q d} + 1$ and $u\subb, u\subt$ by $\floor{\log_q h} + 1$. The other term comes from Theorem~\ref{univariate orbit size upper bound}.

\section{Subspaces of univariate polynomials}\label{section: univariate conjectures}

In this section, we give two conjectures that were discovered in earlier attempts to prove results in Section~\ref{section: orbit size univariate} bounding the period length of an orbit under the linear operator $\lambda_0$.
They were not needed in the end, but they are interesting in their own right since they identify additional structure in $\lambda_0$.
For a polynomial $R \in \field[z]$, define $\lambda_0(S) = \Lambda_0\!\paren{S R^{q - 1}}$ as in Equation~\eqref{univariate little lambda definition}.

The first conjecture implies the conclusion of Theorem~\ref{square-free orbit size upper bound - positive degree}, given in Proposition~\ref{square-free orbit size upper bound - positive degree - conjecture conclusion} below.
For an integer $m \geq 1$, consider the set of polynomials $S \in \field[z]$ such that $\deg S \leq \deg R$ and $\lambda_0^m(S) = S$.
This set forms a vector space.

\begin{conjecture}\label{dimensions conjecture}
Let $R \in \field[z]$ such that $\deg R \geq 1$ and $R$ is not divisible by $z$.
Let $R = c R_1^{e_1} \cdots R_k^{e_k}$ be its factorization into irreducibles.
For every divisor $m$ of $\lcm(\deg R_1, \dots, \deg R_k)$, the vector space
\begin{equation}\label{univariate subspace}
	\{S \in \field[z] : \text{$\deg S \leq \deg R$ and $\lambda_0^m(S) = S$}\}
\end{equation}
has dimension $1 + \sum_{i = 1}^k \gcd(m, \deg R_i)$.
\end{conjecture}

In particular, the exponents $e_i$ do not affect the dimension.

\begin{proposition}\label{square-free orbit size upper bound - positive degree - conjecture conclusion}
Let $R \in \field[z]$ be a nonzero square-free polynomial such that $\deg R \geq 1$ and $R$ is not divisible by $z$.
Let $R = c R_1 \cdots R_k$ be its factorization into irreducibles, and let $\ell = \lcm(\deg R_1, \dots, \deg R_k)$.
Conjecture~\ref{dimensions conjecture} implies that $\lambda_0^\ell(S) = S$ for all $S \in \field[z]$ with $\deg S \leq \deg R$.
\end{proposition}

\begin{proof}
For $m = \ell$, Conjecture~\ref{dimensions conjecture} states that the vector space~\eqref{univariate subspace} has dimension
\[
	1 + \sum_{i = 1}^k \gcd(\ell, \deg R_i)
	= 1 + \sum_{i = 1}^k \deg R_i
	= 1 + \deg R,
\]
so in fact it is the entire space $\{S \in \field[z] : \deg S \leq \deg R\}$.
\end{proof}

A natural question is whether we can write down an explicit basis of the vector space~\eqref{univariate subspace}.
For $m = 1$, Conjecture~\ref{dimensions conjecture} implies that the subspace of fixed points has dimension $k + 1$.
The next conjecture provides a basis of this subspace, for certain polynomials $R$.
One basis element is $R$ itself, since $\lambda_0(R) = \Lambda_0(R^q) = R$ by Proposition~\ref{Cartier}.
We get $k$ additional basis elements from the following operation.
For a polynomial $S = \sum_{j = 0}^s c_j z^j \in \field[z]$ where $c_s \neq 0$, define $\Delta(S) = \sum_{j = 0}^s (s - j) c_j z^j$.
Equivalently, $\Delta(S) = z^{s - 1} \frac{\dd}{\dd w} (w^s S)$ where $w = \frac{1}{z}$.
From this it follows that $\Delta$ is a derivation.
That is, $\Delta(S T) = \Delta(S) T + S \Delta(T)$ for all $S, T \in \field[z]$.

\begin{conjecture}\label{fixed points basis conjecture}
Let $R \in \field[z]$ such that $\deg R \geq 1$.
Let $R = c R_1^{e_1} \cdots R_k^{e_k}$ be its factorization into irreducibles.
For each $i \in \{1, 2, \dots, k\}$, the polynomial $R_1^{e_1} \cdots R_{i - 1}^{e_{i - 1}} \Delta(R_i^{e_i}) R_{i + 1}^{e_{i + 1}} \cdots R_k^{e_k}$ is a fixed point of $\lambda_0$.
Moreover, if $R$ is not divisible by $z$ and $e_i \nequiv 0 \mod p$ for all $i$, where $p$ is the characteristic of $\field$, then these $k$ fixed points, along with $R$, are linearly independent.
\end{conjecture}

If $R$ is divisible by $z$, then we don't get a basis element because $\Delta(z) = 0$.
Similarly, if $e_i \equiv 0 \mod p$ then $\Delta(R_i^{e_i}) = e_i R_i^{e_i - 1} \Delta(R_i) = 0$.

Conjecture~\ref{fixed points basis conjecture} implies that $\Delta(R)$ is a fixed point of $\lambda_0$, since we can use the fact that $\Delta$ is a derivation to write $\Delta(R)$ as a sum of fixed points.

For $m \geq 2$, it would be interesting to know how to extend the basis of fixed points in Conjecture~\ref{fixed points basis conjecture} to a basis of the vector space~\eqref{univariate subspace}.

\section*{Acknowledgment}

The third author thanks Boris Adamczewski for several helpful discussions.

\section*{Appendix}

The next few pages contain the tables and figures mentioned in Section~\ref{section: numeric examples}.

\small

\begin{table}[h]
	\raggedright
	$q = 2$:
	\[
		\begin{array}{c|c|c|c|c|c}
			h & d & P & \text{aut.\ size} & q^{h d} & \text{bound} \\
			\hline
			\hline
			1 & 1 & y + x & 3 & 2 & 6 \\
			2 & 1 & (x^2 + x + 1) y + x^2 & 6 & 4 & 11 \\
			3 & 1 & (x^3 + x + 1) y + x^3 & 11 & 8 & 16 \\
			4 & 1 & (x^4 + x + 1) y + x^4 & 20 & 16 & 27 \\
			5 & 1 & (x^5 + x^3 + 1) y + x^5 & 37 & 32 & 45 \\
			6 & 1 & (x^6 + x + 1) y + x^6 & 70 & 64 & 77 \\
			7 & 1 & (x^7 + x + 1) y + x^7 & 135 & 128 & 147 \\
			8 & 1 & (x^8 + x^7 + x^2 + x + 1) y + x^8 & 264 & 256 & 280 \\
			9 & 1 & (x^9 + x^5 + 1) y + x^9 & 521 & 512 & 541 \\
			10 & 1 & (x^{10} + x^3 + 1) y + x^{10} & 1034 & 1024 & 1063 \\
			\hline
			1 & 2 & x y^2 + (x + 1) y + x & 7 & 4 & 10 \\
			2 & 2 & x^2 y^2 + (x^2 + x + 1) y + x^2 & 14 & 16 & 25 \\
			3 & 2 & (x^3 + x^2 + 1) y^2 + (x^3 + 1) y + x & 68 & 64 & 93 \\
			4 & 2 & (x^4 + x + 1) y^2 + (x^4 + x^2 + x + 1) y + x & 252 & 256 & 311 \\
			5 & 2 & (x^5 + x^3 + 1) y^2 + (x^5 + x + 1) y + x & 1052 & 1024 & 1191 \\
			6 & 2 & (x^6 + x^5 + 1) y^2 + (x^6 + x^2 + x + 1) y + x & 4062 & 4096 & 4423 \\
			7 & 2 & (x^7 + x + 1) y^2 + (x^7 + x^4 + x^3 + x + 1) y + x & 16424 & 16384 & 17287 \\
			\hline
			1 & 3 & x y^3 + y^2 + (x + 1) y + x & 11 & 8 & 18 \\
			2 & 3 & (x^2 + x + 1) y^3 + y^2 + (x^2 + 1) y + x^2 + x & 61 & 64 & 93 \\
			3 & 3 & (x^3 + x + 1) y^3 + y^2 + (x^3 + x^2 + x + 1) y + x^3 + x^2 & 533 & 512 & 613 \\
			4 & 3 & (x^4 + x + 1) y^3 + y^2 + (x^4 + 1) y + x^4 + x^3 + x & 4213 & 4096 & 4871 \\
			\hline
			1 & 4 & (x + 1) y^4 + y^2 + (x + 1) y + x & 20 & 16 & 33 \\
			2 & 4 & (x^2 + x + 1) y^4 + y^3 + (x^2 + x + 1) y + x^2 + x & 216 & 256 & 358 \\
			3 & 4 & (x^3 + x + 1) y^4 + y^3 + (x^3 + 1) y + x^2 + x & 3956 & 4096 & 4870 \\
			\hline
			1 & 5 & (x + 1) y^5 + (x + 1) y^2 + y + x & 37 & 32 & 67 \\
			2 & 5 & (x^2 + x + 1) y^5 + y^4 + y^3 + x^2 y^2 + y + x^2 + x & 889 & 1024 & 1510 \\
			3 & 5 & (x^3 + x^2 + 1) y^5 + y^4 + x^3 y^2 + (x + 1) y + x^3 + x^2 + x & 43913 & 32768 & 48134
		\end{array}
	\]
	$q = 3$:
	\[
		\begin{array}{c|c|c|c|c|c}
			h & d & P & \text{aut.\ size} & q^{h d} & \text{bound} \\
			\hline
			\hline
			1 & 1 & (x + 1) y + x & 4 & 3 & 7 \\
			2 & 1 & (2 x^2 + x + 1) y + x^2 & 11 & 9 & 14 \\
			3 & 1 & (x^3 + 2 x + 1) y + x^3 & 30 & 27 & 35 \\
			4 & 1 & (2 x^4 + x + 1) y + x^4 & 85 & 81 & 90 \\
			5 & 1 & (x^5 + 2 x + 1) y + x^5 & 248 & 243 & 254 \\
			6 & 1 & (2 x^6 + x + 1) y + x^6 & 735 & 729 & 740 \\
			\hline
			1 & 2 & (x + 1) y^2 + y + x & 9 & 9 & 14 \\
			2 & 2 & (x^2 + x + 2) y^2 + y + x^2 & 79 & 81 & 90 \\
			3 & 2 & (x^3 + x^2 + 2 x + 1) y^2 + y + x^3 + x & 727 & 729 & 788 \\
			4 & 2 & (x^4 + x^3 + 2) y^2 + y + x^4 + x & 6533 & 6561 & 6728
		\end{array}
	\]
	\caption{Polynomials in $\field[x, y]$ achieving the maximum unminimized automaton size for given values of $q$, $h$, and $d$, for comparison with the bound in Theorem~\ref{kernel size upper bound}.}
	\label{automaton size table}
\end{table}

\begin{table}[h]
	\raggedright
	$q = 2$:
	\[
		\begin{array}{c|c|c|c|c}
			h & d & P & \text{orbit size} & \text{bound} \\
			\hline
			\hline
			1 & 1 & y + x & 2 & 4 \\
			2 & 1 & y + x^2 & 3 & 7 \\
			3 & 1 & (x^3 + x + 1) y + x & 4 & 8 \\
			4 & 1 & (x^4 + x^3 + 1) y + x & 5 & 11 \\
			5 & 1 & (x^5 + x + 1) y + x & 7 & 13 \\
			6 & 1 & (x^6 + x^3 + 1) y + x & 7 & 13 \\
			7 & 1 & (x^7 + x^6 + x^2 + x + 1) y + x & 13 & 19 \\
			8 & 1 & (x^8 + x^3 + 1) y + x & 16 & 24 \\
			9 & 1 & (x^9 + x^2 + 1) y + x & 21 & 29 \\
			10 & 1 & (x^{10} + x^6 + x^3 + x^2 + 1) y + x & 31 & 39 \\
			\hline
			1 & 2 & x y^2 + (x + 1) y + x & 3 & 6 \\
			2 & 2 & x^2 y^2 + (x^2 + x + 1) y + x^2 & 6 & 9 \\
			3 & 2 & (x^3 + x^2 + 1) y^2 + (x^3 + 1) y + x & 12 & 29 \\
			4 & 2 & (x^4 + x^2 + x) y^2 + (x^4 + x + 1) y + x^4 & 25 & 55 \\
			5 & 2 & (x^5 + x^3 + 1) y^2 + (x^5 + x + 1) y + x & 60 & 167 \\
			6 & 2 & (x^6 + x^4 + x) y^2 + (x^5 + x + 1) y + x & 61 & 327 \\
			7 & 2 & (x^7 + x + 1) y^2 + (x^7 + x^4 + x^3 + x + 1) y + x & 168 & 903 \\
			8 & 2 & (x^8 + x^3 + 1) y^2 + (x^8 + x^7 + x^2 + x + 1) y + x^8 & 240 & 3849 \\
			\hline
			1 & 3 & (x + 1) y^3 + y^2 + (x + 1) y + x & 7 & 10 \\
			2 & 3 & (x^2 + 1) y^3 + y^2 + (x^2 + x + 1) y + x^2 & 14 & 29 \\
			3 & 3 & (x^3 + x + 1) y^3 + x^2 y^2 + y + x^3 & 85 & 101 \\
			4 & 3 & (x^4 + x^2 + x + 1) y^3 + y^2 + y + x^3 + x^2 + x & 373 & 775 \\
			5 & 3 & (x^5 + x^2 + 1) y^3 + y^2 + (x^5 + x^3 + 1) y + x^5 + x^4 + x & 7621 & 7687 \\
			\hline
			1 & 4 & (x + 1) y^4 + y^2 + (x + 1) y + x & 12 & 17 \\
			2 & 4 & x^2 y^4 + (x^2 + x + 1) y^3 + (x + 1) y + x^2 + x & 26 & 102 \\
			3 & 4 & (x^3 + x^2 + x + 1) y^4 + (x^3 + x + 1) y^3 + (x^2 + 1) y + x^3 & 375 & 774 \\
			4 & 4 & (x^4 + 1) y^4 + (x^4 + x^3 + 1) y^3 + (x + 1) y + x^4 + x^3 & 5209 & 6151 \\
			\hline
			1 & 5 & (x + 1) y^5 + (x + 1) y^2 + y + x & 21 & 35 \\
			2 & 5 & (x^2 + 1) y^5 + y^4 + x y^3 + x^2 y^2 + (x + 1) y + x^2 + x & 122 & 486 \\
			3 & 5 & (x^3 + x^2 + 1) y^5 + y^4 + x^3 y^2 + (x + 1) y + x^3 + x^2 + x & 15241 & 15366
		\end{array}
	\]
	$q = 3$:
	\[
		\begin{array}{c|c|c|c|c}
			h & d & P & \text{orbit size} & \text{bound} \\
			\hline
			\hline
			1 & 1 & y + x & 2 & 4 \\
			2 & 1 & (x^2 + 1) y + x & 3 & 5 \\
			3 & 1 & (x^3 + 2 x^2 + 1) y + x & 4 & 8 \\
			4 & 1 & (2 x^4 + x + 1) y + x & 5 & 9 \\
			5 & 1 & (x^5 + 2 x^2 + 1) y + x & 7 & 11 \\
			6 & 1 & (x^6 + x + 1) y + x^2 & 7 & 11 \\
			7 & 1 & (x^7 + 2 x^3 + 1) y + x & 13 & 17 \\
			\hline
			1 & 2 & x y^2 + y + x & 3 & 5 \\
			2 & 2 & (x^2 + 1) y^2 + y + x^2 + x & 7 & 9 \\
			3 & 2 & (x^3 + 2 x + 2) y^2 + y + x^2 + x & 25 & 59 \\
			4 & 2 & (x^3 + 2 x + 2) y^2 + y + x^4 & 79 & 167
		\end{array}
	\]
	\caption{Polynomials in $\field[x, y]$ for which the initial state achieves the maximum orbit size under $\lambda_{0, 0}$ for given values of $q$, $h$, and $d$. The final column contains the value of $q^{(h - 1) (d - 1)} \Landaulcm(h, d, d) + \floor{\log_q h} + \floor{\log_q \max(h, d)} + 3$ from Theorem~\ref{kernel size upper bound}.}
	\label{orbit size table}
\end{table}

\begin{figure}[h]
	\begin{subfigure}{0.40\textwidth}
		\includegraphics[width = \textwidth]{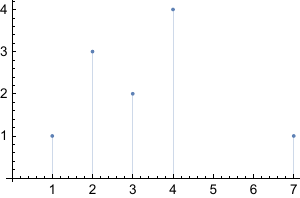}
		\caption*{$q = 2$, $h = 1$}
	\end{subfigure}
	\qquad \qquad
	\begin{subfigure}{0.40\textwidth}
		\includegraphics[width = \textwidth]{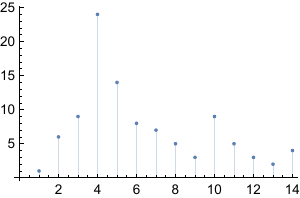}
		\caption*{$q = 2$, $h = 2$}
	\end{subfigure} \\
	\begin{subfigure}{0.40\textwidth}
		\includegraphics[width = \textwidth]{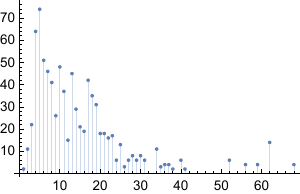}
		\caption*{$q = 2$, $h = 3$}
	\end{subfigure}
	\qquad \qquad
	\begin{subfigure}{0.40\textwidth}
		\includegraphics[width = \textwidth]{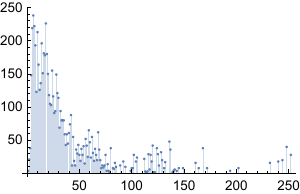}
		\caption*{$q = 2$, $h = 4$}
	\end{subfigure} \\
	\begin{subfigure}{0.40\textwidth}
		\includegraphics[width = \textwidth]{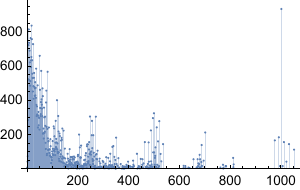}
		\caption*{$q = 2$, $h = 5$}
	\end{subfigure}
	\qquad \qquad
	\begin{subfigure}{0.40\textwidth}
		\includegraphics[width = \textwidth]{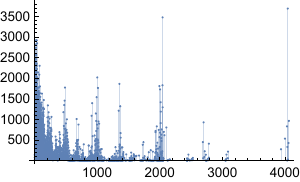}
		\caption*{$q = 2$, $h = 6$}
	\end{subfigure} \\
	\begin{subfigure}{0.40\textwidth}
		\includegraphics[width = \textwidth]{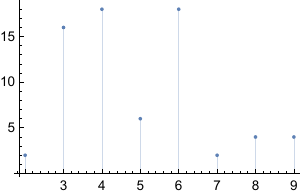}
		\caption*{$q = 3$, $h = 1$}
	\end{subfigure}
	\qquad \qquad
	\begin{subfigure}{0.40\textwidth}
		\includegraphics[width = \textwidth]{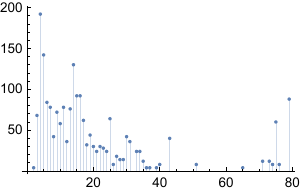}
		\caption*{$q = 3$, $h = 2$}
	\end{subfigure} \\
	\begin{subfigure}{0.40\textwidth}
		\includegraphics[width = \textwidth]{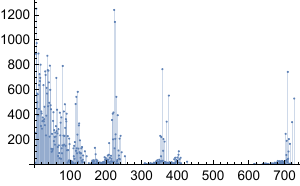}
		\caption*{$q = 3$, $h = 3$}
	\end{subfigure}
	\qquad \qquad
	\begin{subfigure}{0.40\textwidth}
		\includegraphics[width = \textwidth]{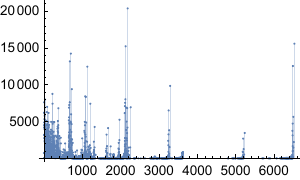}
		\caption*{$q = 3$, $h = 4$}
	\end{subfigure}
	\caption{Number of polynomials (vertical axis) with degree $d = 2$ that produce unminimized automata with a given size (horizontal axis). The top six plots are for $q = 2$ and vary $h \in \{1, 2, \dots, 6\}$. In the bottom four, $q = 3$ and $h \in \{1, 2, 3, 4\}$.}
	\label{automaton size figure}
\end{figure}

\end{document}